\documentclass[11pt]{article}
\usepackage{color}
\usepackage{amsmath,amsthm,amssymb,amsfonts}
\usepackage{epsfig}
\usepackage{graphicx}

\usepackage{todonotes}

\usepackage{bm}
\usepackage{dsfont}
\numberwithin{equation}{section}

\newtheorem{theorem}{\bf Theorem}[section]
\newtheorem{lemma}[theorem]{\bf Lemma}

\newtheorem{definition}[theorem]{\bf Definition}
\newtheorem{remark}[theorem]{\bf Remark}

\def\R{\mathbb R}
\def\N{\mathbb N}

\def\e{\varepsilon}
\def\trait (#1) (#2) (#3){\vrule width #1pt height #2pt depth #3pt}
\def\fin{\hfill\trait (0.1) (5) (0) \trait (5) (0.1) (0) \kern-5pt
\trait (5) (5) (-4.9) \trait (0.1) (5) (0)}

\newcommand{\be}{\begin{equation}}
\newcommand{\ee}{\end{equation}}
\newcommand{\baa}{\begin{array}}
\newcommand{\eaa}{\end{array}}
\newcommand{\ba}{\begin{eqnarray}}
\newcommand{\ea}{\end{eqnarray}}

\newcommand{\ban}{\begin{eqnarray*}}
\newcommand{\ean}{\end{eqnarray*}}

\def\e{\varepsilon}

\def\Ac{{\cal A}}
\def\Tc{{\cal T}}

\def\VFp{\mathbf{U}^+}
\def\VFm{\mathbf{U}^-}

\newcommand{\dtau}{\,\mathrm{d}\tau}
\newcommand{\dalpha}{\,\mathrm{d}\alpha}
\newcommand{\ds}{\,\mathrm{d}s}

\def\apg{\left\{}
\def\chg{\right\}}
\def\apt{\left(}
\def\cht{\right)}

\def\a{a}

\def\A{{\cal A}}

\def\Ta{{\cal T}_{x,t}}
\def\Taineps{{\cal T}_{x_\eps,t_\eps}}
\def\Taeps{{\cal T}_{x,t}^{\eps}}
\def\Tan{{\cal T}_{x_n,t_n}}
\def\Treg{{\cal T}^{\rm reg}_{x,t}}
\def\Tregineps{{\cal T}^{\rm reg}_{x_\eps,t_\eps}}
\def\Tregeps{{\cal T}^{\mathrm{reg},\eps}_{x,t}}
\def\HTreg{{H}^{\rm reg}_T}

\def\HT{{H}_T}

\def\Aoreg{A_0^{\rm reg}}
\def\Aoregxt{A_0^{\rm reg}(x,t)}

\def\Esing{E_{\rm sing}^{\eta}}

\def\eps{\varepsilon}
\renewcommand{\H}{\mathcal{H}}

\newcommand{\mB}{\mathcal{B}}
\newcommand{\cob}{\overline{\mathop{\rm co}}}

\newcommand{\trajxt}{X_{x,t}}
\newcommand{\trajxntn}{X_{x_n,t_n}}

\def\dottrajxt{\dot{X}_{x,t}}

\def\trajyali{Y_{x,t}^{i}}

\def\dottrajyali{{\dot Y}_{x,t}^{i}}
\def\trajyxki{Y_{x_k,t}^{i}}

\newcommand{\hyp}[2]{$(\mathbf{H}_{#1}^{#2})$}

\newcommand{\Gamij}{\Gamma_{\{i,j\}}}

\newcommand{\dH}{\mathrm{d}_\H}

\newcommand{\Ei}{\mathcal{E}_{i}}

\newcommand{\psg}[2]{\big\langle #1,#2\big\rangle}
\newcommand{\E}{\mathcal{E}}
\newcommand{\Q}[1]{{Q}^{(#1)}}
\DeclareMathOperator{\dist}{\mathrm{dist}}

\renewcommand{\H}{\mathcal{H}}

\def\HT{H_T}

\def\bH{b_{\H}}

\def\lH{\l_{\H}}

\def\HTreg{{ H}^{\rm reg}_T}

\def\nor{\mathbf{n}}

\def\HHm{\mathbb{H}^-}
\def\HHp{\mathbb{H}^+}
\def\tHHm{\tilde{\mathbb{H}}^-}

\begin{document}

\title{\bf A Bellman approach for regional optimal control problems in $\R^N$}
\author{G.~Barles\thanks{Laboratoire de
  Math\'ematiques et Physique Th\'eorique (UMR CNRS 6083), F\'ed\'eration Denis
  Poisson (FR CNRS 2964), Universit\'e Fran\c{c}ois Rabelais, Parc de Grandmont,
  37200 Tours, France.\newline \indent This work was partially supported by the ANR HJnet ANR-12-BS01-0008-01 and by EU under
  the 7th Framework Programme Marie Curie Initial Training Network
  ``FP7-PEOPLE-2010-ITN'', SADCO project, GA number 264735-SADCO.} 
   \and A.~Briani${}^{*}$ \and E.~Chasseigne${}^{*}$ } \maketitle
\begin{abstract}
This article is a continuation of a previous work where we studied infinite
horizon control problems for which the dynamic, running cost and control
space may be different in two half-spaces of some Euclidian space
$\R^N$. In this article we extend our results in several directions:
$(i)$ to more general domains; $(ii)$ by considering finite horizon control
problems; $(iii)$ by weakening the controlability assumptions. We use a Bellman approach and our main
results are to identify the right Hamilton-Jacobi-Bellman Equation
(and in particular the right conditions to be put on the interfaces
separating the regions where the dynamic and running cost are different)
and to provide the maximal and minimal solutions, as well as conditions for
uniqueness. We also provide stability results for such equations.

\end{abstract}

 \noindent {\bf Key-words}: Optimal control, discontinuous dynamic, Bellman Equation, viscosity solutions.
\\
{\bf AMS Class. No}:
49L20,   
49L25,   
35F21. 
\section{Introduction}

This article is a continuation of \cite{BBC1} where we studied infinite
horizon control problems for which the dynamic, running cost and control
space may be different in two half-spaces of some  Euclidian space
$\R^N$. This study was made through the Bellman approach and our main
results where to identify the right Hamilton-Jacobi-Bellman Equation
(and in particular the right conditions to be put on the hyperplane
separating the regions where the dynamic and running cost are different)
and to provide the maximal and minimal solutions, as well as conditions for
uniqueness. The aim of the present paper is three-fold: $(i)$ to extend these
results to more general domains; $(ii)$ to consider also finite horizon control
problems; $(iii)$ last but not least, to weaken the controlability
assumption made in \cite{BBC1}. We also emphasize the stability properties for such equations
which are a little bit different from the classical ones.

To be more specific, we recall that, in the classical theory (see for example Lions \cite{L}, Fleming \&
Soner \cite{fs}, Bardi \& Capuzzo Dolcetta \cite{BCD}),
Hamilton-Jacobi-Bellman Equation for finite horizon control problems in
the whole space $\R^N$ have the form
\begin{equation}\label{eq:gen}
      u_t+H(x,t,Du) = 0  \quad\hbox{ in  }\R^N\times (0,T)\,, \\
\end{equation} 
where the Hamiltonian $H$ is typically given by
\begin{equation}\label{def:H.0}
	H(x,t,p):=\sup_{ \alpha\in A} \big\{-b(x,t,\alpha) \cdot p -
	    l(x,t,\alpha) \big\}\,.
\end{equation} 
The control space $A$ is assumed to be compact, the dynamic $b$ and running cost $l$ are supposed to be continuous functions which are Lipschitz continuous in $x$, so that $H$
is continuous and has suitable properties ensuring existence and
uniqueness of a solution to~\eqref{eq:gen}.

In this paper, as we already mentioned  above, we have different dynamics and running costs in different regions. In other words, the functions $b$ and $l$ are no longer assumed to be continous anymore when crossing the boundaries of the different regions,
which implies that the Hamiltonian $H$ in~\eqref{def:H.0} also
presents discontinuities. Hence, getting suitable comparison and
uniqueness results for~\eqref{eq:gen} in this setting is not obvious at
all and the aim of this paper is to give precise answers to these
questions.

To be more precise, we are going to decompose $\R^N$ using a collection
$(\Omega_i)_{i\in I}$ of regular open subsets of $\R^N$ such that each
point $x\in \R^N$ either lies inside one (and only one) $\Omega_i$, or
is located on the boundary of exactly two sets $\Omega_i$. Because of
the (regularity) assumptions we are going to use, we can in fact reduce
this collection to two domains $\Omega_1, \Omega_2$ : we refer to
Section~\ref{sec:FRE} for comments on this reduction. More precisely we
assume that \begin{itemize} 
    \item[\hyp{\Omega}{}] $\R^N=\Omega_1\cup\Omega_2\cup\H$ with
	  $\Omega_1\cap\Omega_2=\emptyset$ and $\H=\partial
	  \Omega_1=\partial\Omega_2$ is a $W^{2,\infty}$-hypersurface in $\R^N$.
\end{itemize}

A consequence of this assumption is the following : if $\dH(\cdot)$ denotes the signed distance function to $\H$ which is positive in $\Omega_1$ and negative in $\Omega_2$, then $\dH$ is $W^{2,\infty}$ in a neighborhood of $\H$. Moreover, for $x\in\H$, $D\dH (x)=-\mathbf{n}_1 (x)=\mathbf{n}_2 (x)$ where, for $i=1,2$, $\mathbf{n}_i(x)$ is the unit normal vector to $\partial \Omega_i$ pointing outwards $\Omega_i$. We will use the notation $-\mathbf{n}_1(x)$ or $\mathbf{n}_2(x)$ for the gradient of $\dH$ at $x$, even if $x$ does not belong to $\H$.
 
In each $\Omega_i$ ($i=1,2$), we have a ``classical'' finite-horizon control
problem and the equation can be written as
\begin{equation}\label{Bellman-Om}
\begin{array}{cc} 
  u_t+H_i(x,t,Du) = 0     &   \hbox{ in   }\Omega_i \times (0,T)\,, \\
\end{array} \end{equation} 
for some $T>0$, where $H_i$ is given by
\begin{equation}  \label{def:Ham} 
  H_i(x,t,p):=\sup_{ \alpha_i \in A_i} \apg -b_i(x,t,\alpha_i) \cdot p
  - l_i(x,t,\alpha_i) \chg\,.
\end{equation} 
The $b_i, l_i$ are at least continuous functions defined on $\overline{\Omega_i}\times (0,T) \times A_i$, the control space $A_i$ being compact metric spaces; precise assumptions will be given later on.

Of course, one has to write down an equation on the whole space $\R^N$
(and in particular on $\H$) and this can be done using viscosity
solutions' theory (\cite{Son}, \cite{Ba}, \cite{BCD}). One can consider
Equation~\eqref{eq:gen} with $H=H_i$ on $\Omega_i$ and use Ishii's
definition  of viscosity solutions for discontinuous Hamiltonians (cf.
\cite{Idef}) which reads
\begin{align*}(u^*)_t + H_*(x,t,Du^*) &= 0  \quad\hbox{ in  }\R^N\times
(0,T)\quad \text{for subsolutions }u\\
 \text{and}\quad (v_*)_t +H^*(x,t,Dv_*) &= 0  
 \quad\hbox{ in  }\R^N\times (0,T)\quad \text{for supersolutions $v$}\,, 
\end{align*}
where the ``upper-star'' denotes the upper semi-continuous envelope while the
``lower-star'' denotes the lower semi-continuous envelope.
Following this means that we have to complement Equations~\eqref{Bellman-Om} by
\begin{align} \min\{ u_t+ H_1(x,t,Du),  u_t+H_2(x,t,Du)\}\leq0  \quad\hbox{on   } \H \times (0,T)\;, \label{Bellman-H-sub} \\ 
  \max\{ u_t+ H_1(x,t,Du), u_t+H_2(x,t,Du)\}\geq 0   \quad\hbox{on   } \H \times (0,T) \; .
  \label{Bellman-H-sup}  
  \end{align}

In order to present our results and to compare them with those of \cite{BBC1}, we are going to describe the main contributions of \cite{BBC1} and the improvements/additional results of the present work. We first point out that the question we address in \cite{BBC1} (and also here) is to investigate the uniqueness properties for \eqref{eq:gen}
or equivalently \eqref{Bellman-Om}-\eqref{Bellman-H-sub}-\eqref{Bellman-H-sup}. The reason why we started to study the question in that way and why we insist on \eqref{Bellman-Om}-\eqref{Bellman-H-sub}-\eqref{Bellman-H-sup} is because of the stability properties of \eqref{Bellman-Om}-\eqref{Bellman-H-sub}-\eqref{Bellman-H-sup} : any approximation of the problem converges  to a solution of \eqref{Bellman-Om}-\eqref{Bellman-H-sub}-\eqref{Bellman-H-sup} and it is, in any case, important to understand the structure of the solutions of \eqref{Bellman-Om}-\eqref{Bellman-H-sub}-\eqref{Bellman-H-sup}.

The first result of \cite{BBC1} was to identify the maximal subsolution (and solution) and the minimal supersolution (and solution) of \eqref{Bellman-Om}-\eqref{Bellman-H-sub}-\eqref{Bellman-H-sup}. Both are value functions of suitable optimal control problems and the difference between them comes from the ``admissible'' strategies which can be used on the interface $\H$ ($\H$ was an hyperplane in \cite{BBC1}). A notion of ``regular'' and ``singular'' strategies is introduced and while, for the maximal solution $\VFp$, only the ``regular'' strategies are allowed, both ``regular'' and ``singular'' strategies can be used for the minimal solution $\VFm$. Roughly speaking, the whole set of ``regular'' and ``singular'' strategies are those which are obtained by an approach of the dynamic and cost via differential inclusions, i.e. by using on $\H$ any convex combination of the dynamics and costs in $\Omega_1$ and $\Omega_2$. ``Regular'' strategies are those for which $b_1$ and $b_2$ are pointing respectively outside  $\Omega_1$ and $\Omega_2$. The main difference between ``regular'' and ``singular'' strategies is that the ``regular'' ones are included in the formulation of\eqref{Bellman-Om}-\eqref{Bellman-H-sub}-\eqref{Bellman-H-sup}, while this is not the case for ``singular'' ones.

We refer the reader to Section~\ref{sec:def-ass} for the description of these different control problems and in particular of the two different value functions $\VFm$ and  $\VFp$, with the (classical) assumptions we are going to use. Of course, we give a precise definition of ``regular'' and ``singular'' strategies. To our point of view, there is no criterion to declare one of these value functions more natural than the other and therefore we pay the same attention to both.

In order to obtain this complete description, we have to do a double work : on one hand, we have to show that $\VFm$ and  $\VFp$ are solutions of \eqref{Bellman-Om}-\eqref{Bellman-H-sub}-\eqref{Bellman-H-sup} and, maybe, to obtain additional viscosity solutions inequalities on $\H$. This is indeed the case for $\VFm$ for which taking into account ``singular'' strategies is translated into an additional subsolution inequality on $\H$, but not for $\VFp$, which partially justify the above sentence claiming that ``regular'' strategies are included in the formulation of \eqref{Bellman-Om}-\eqref{Bellman-H-sub}-\eqref{Bellman-H-sup} (see also Theorem~\ref{teo:sotto}). Then we have to study the properties of general sub and supersolutions of \eqref{Bellman-Om}-\eqref{Bellman-H-sub}-\eqref{Bellman-H-sup} and more particularly on $\H$. Of course, and this is rather classical, we have connect these sub and supersolutions properties with sub or super-optimality principles. This is done in Section~\ref{pss}.

The difference here with \cite{BBC1} is that $\VFm$, $\VFp$ are not necessarily continuous since, at the same time, we have weakened the controlability assumption and we consider finite horizon control problems. The first consequence is that the connections with the Bellman Equation (\ref{Bellman-Om})-(\ref{Bellman-H-sub})-(\ref{Bellman-H-sup}) in Section~\ref{sect:edp} has to be stated in terms of discontinuous viscosity solutions (cf. Theorem~\ref{teo:HJ}). Then, still in Section~\ref{sect:edp}, we provide properties, satisfied either by $\VFp$ or by general sub and supersolutions which play a key role in order to obtain comparison results.

The next step consists in studying uniqueness-comparison properties. Of course, there is no general comparison result for (\ref{Bellman-Om})-(\ref{Bellman-H-sub})-(\ref{Bellman-H-sup}) since, in general, we have more than one solution ($\VFm$ and $\VFp$) but it turns out that, if we add a viscosity subsolution inequality on $\H$ (related, as we already mentioned it above, to singular strategies), then not only $\VFm$ becomes the only solution of this new problem but we have a full Strong Comparison Result for this new problem (i.e. a comparison result between discontinuous sub and supersolutions). This allows us to perform all the classical pde arguments in the $\VFm$ case. On the contrary, we were unable to find a pde characterization of $\VFp$ and all the proof requires optimal control arguments. This explains why we (unfortunately) have to double a lot of proofs since those for $\VFm$ and $\VFp$ have to use completely different arguments.

Compared to \cite{BBC1}, we have modified the strategy of the comparison
proofs by emphasizing the role of a ``local comparison result'' which is
given in the Appendix. There are several reasons to do so : such local
results are useful for applications, for example in homogenization
problems which we consider in a forthcoming work with N. Tchou
\cite{BBCT}; in such applications the use of the perturbed test-function
of L. C. Evans \cite{Evans1989,Evans1992} requires (or is far more
simpler with) such local comparison results. On the other hand we have
to handle, at the same time, a more complex geometry than in \cite{BBC1}
and a weaker controlability assumption (which implies that the sub
solutions are not automatically Lipschitz continuous) and to argue
locally allow to flatten the interface and use a double regularization
procedure on the subsolutions in the tangent variables, first by
sup-convolution to reduce to the Lipschitz continuous case and then by
usual mollification. Here it is worth pointing out the double role of
the ``controlability in the normal direction'' on $\H$: first,
technically, this allows to perform the sup-convolution procedure in the tangent variables only by, roughly speaking, inducing a control of the normal derivatives of the solution by the tangent derivatives. Then the same argument implies that a subsolution which is Lipschitz continuous in the tangent variable is Lipschitz continuous with respect to all variables and this is precisely the case for the subsolution obtained by sup-convolution.

Finally, in \cite{BBC1}, we did not really address the question of the
stability properties, despite we provide few partial results.  In
Section~\ref{sect:stability}, we study them more systematically. As we
already mentioned above, the results and the proofs for $\VFm$ and
$\VFp$ are completely different. For the problem satisfied by $\VFm$, it
is (almost) a ``classical'' stability result proved by (almost)
``classical'' arguments, but contrarily to the standard results in
viscosity solutions' theory, we face a difficulty because of the
discontinuity on $\H$, difficulty which is solved in an unusual way by
the controlability assumption in the normal direction. On the contrary,
for the problem satisfied by $\VFp$, we prove the stability of
controlled trajectories and costs, a rather delicate result since we
have to show that the limit of trajectories with ``regular'' strategies
is a trajectory wich can be represented by a ``regular'' strategy. In this second case, we have no pde approach and therefore this is the only kind of results we may hope to have.

Finally Section~\ref{sec:FRE} is devoted to describe several extensions, in particular to multi-domains problems in which the domains may also depend on time.

There are more and more articles on Hamilton-Jacobi-Bellman Equations or
control problems on multi-domains (also called stratified domains). We
start by recalling the pioneering work by Dupuis \cite{Du} who uses
similar methods to construct a numerical method for a calculus of
variation problem with discontinuous integrand. Problems with a
discontinuous running cost were addressed by either Garavello and
Soravia \cite{GS1,GS2}, or Camilli and Siconolfi \cite{CaSo} (even in an
$L^\infty$-framework) and Soravia \cite{So}. To the best of our
knowledge, all the uniqueness results use a special structure of the
discontinuities as in \cite{DeZS,DE,GGR} or an hyperbolic approach as in
\cite{AMV,CR}. Recent works on optimal control problem on stratified
domains are the ones of Bressan and Hong \cite{BrYu} but also Barnard
and Wolenski \cite{BaWo} and Rao and Zidani \cite{RaZi} (who mention a
forthcoming work with Siconolfi \cite{RaSiZi}): in these three last
works,  where the approach is different since they do not
  start from
  (\ref{Bellman-Om})-(\ref{Bellman-H-sub})-(\ref{Bellman-H-sup}) and
  instead write Bellman Equations which are adapted to the dynamic of
  the problem and the geometry of the discontinuities, uniqueness
results are provided by a different method than ours, which
completely relies on control arguments. The advantage of their methods is to allow
them to handle more general stratified domains (non-smooth domains with
multiple junctions) but with more restrictive controlability assumptions
and without the stability results we can provide. We finally remark that
problems on network (see \cite{IMZ}, \cite{ACCT}, \cite{ScCa}) share the
same kind of difficulties: indeed one has to take into account the
junctions as we have to deal with the interface~$\H$.

\smallskip

\textbf{Acknowledgements ---} We would like to thank Nicoletta Tchou for several constructive remarks on the preliminary versions of this paper.

\smallskip

\section{The optimal control problem}\label{sec:def-ass}

\noindent\textsc{The control problem ---} We fix $T>0$ and consider
that, on each domain $\Omega_i$ ($i=1,2$) we have a controlled dynamic given by
$b_i : \overline{\Omega_i}\times[0,T]\times A_i \to \R^N$, where $A_i$
is the compact metric space where the control takes its values. We have
also a running cost $l_i : \overline{\Omega_i}\times[0,T]\times A_i \to
\R$. Throughout the paper, we make the following assumption on the initial
cost:
\begin{itemize}
    \item[\hyp{g}{}] \textit{The function $g$ is bounded and continuous in
	$\R^N$\,.}
\end{itemize}	
Our main assumptions for the control problem are the following. 
\begin{itemize} 
\item[\hyp{\rm C}{1}] \textit{For any $i=1,2$, $A_i$ is a compact metric
	space and $b_i : \overline{\Omega_i}\times[0,T]\times A_i \to
	\R^N$ is a continuous bounded function. More precisely  there
	exists  $M_b >0$, such that for any $x \in \R^N$, $s \in [0,T]$
	and $\alpha_i \in A_i$, $i=1,2$, $$ |b_i(x,s,\alpha_i) | \leq
	M_b \; . $$ Moreover there exists $L_b \in \R$ such that, for
	any $z, z'\in  \overline{\Omega_i}$,
	$s, s'\in[0,T]$ and $\alpha_i \in A_i$, $i=1,2,$} 
    $$|b_i(z,s,\alpha_i)-b_i(z',s',\alpha_i)|\leq
    L_b (|z-z'|+|s-s'|)\; .$$

\item[\hyp{\rm C}{2}]  \textit{For any $i=1,2$, the function $l_i :
	\overline{\Omega_i}\times[0,T]\times A_i  \to \R^N$ is a uniformly
	continuous, bounded function.   More precisely  there exists  
$M_l >0$, such that for any $x \in \R^N$, $s \in [0,T]$ and $\alpha_i
\in A_i$, $i=1,2$,
$$ |l_i(x,s,\alpha_i) | \leq M_l \,. $$ 
Moreover there exists a modulus of continuity $m_l : [0,+\infty) \to [0,+\infty)$ such that, for any $z, z' \in
	\overline{\Omega_i}$, $s,s'\in[0,T]$ and $\alpha_i \in A_i$, $i=1,2,$}
    $$|l_i(z,s,\alpha_i)-l_i(z',s',\alpha_i)|\leq m_l (|z-z'|+|s-s'|)\;  .$$ 

%

\item[\hyp{\rm C}{3}] \textit{For each $i=1,2$, $z \in \overline{\Omega_i}$, and
	$s\in[0,T]$, the set $\big\{
	    \big(b_i(z,s,\alpha_i),l_i(z,s,\alpha_i)\big) : \alpha_i \in
	    A_i \big\}$ is closed and convex. }

\item[\hyp{\rm C}{4}] \textit{There is a $\delta>0$ such that for any
	$i=1,2$, $z\in\H$ and $s\in[0,T]$}
    \begin{equation}\label{cont-ass}
        \mathbf{B}_i(z,s)\cdot\mathbf{n}_i(z)\supset [-\delta,\delta]
     \end{equation} 
     where $\mathbf{B}_i(z,s):= \big\{ b_i(z,s,\alpha_i) : \alpha_i \in
	 A_i \big\}$\,.  
\end{itemize}

Assumption \hyp{\rm C}{1} and \hyp{\rm C}{2} are the classical
hypotheses used in control problems, while \hyp{\rm C}{3} avoids the use of relaxed controls.  Hypothesis \hyp{\rm C}{4}
expresses some controllability condition but only in the normal
direction when the point $x$ belongs to the boundaries shared by the
sets $\Omega_i$. In the sequel, we refer to \hyp{\rm C}{} as the
intersection of all the four hypotheses \hyp{\rm C}{1}--\hyp{\rm C}{4}.

\

\noindent\textsc{Boundary dynamics ---} In order to define the
controlled dynamics and trajectories which may stay for a while on the
common boundary $\H$, we introduce the boundary dynamic as
follows: if $s\in[0,T]$, $z\in\H$ we set
\begin{equation*}
    b_\H\big(z,s,a)=b_\H\big(z,s,(\alpha_1, \alpha_2, \mu)\big):=\mu
    b_1 (z,s,\alpha_1) + (1-\mu)b_2(z,s,\alpha_2)\,,
\end{equation*}
where $\mu \in [0,1]$, $\alpha_1 \in A_1$, $\alpha_2 \in A_2$.  For any
$z\in\H$ and $s\in[0,T]$ we denote by
\begin{equation*}
    A_0(z,s):=\Big\{\a=(\alpha_1, \alpha_2, \mu):
	b_\H\big(z,s,(\alpha_1, \alpha_2, \mu)\big)\cdot
	\mathbf{n}_1(z) = 0\Big\}\,,
\end{equation*}
and the associated cost on $\H$ is
\begin{equation*}
    l_\H(z,s,\a)=l_\H\big(z,s,(\alpha_1, \alpha_2, \mu)\big):=\mu
    l_1 (z,s,\alpha_1) + (1-\mu)l_2(z,s,\alpha_2)\,.
\end{equation*}
Notice that the dynamic and cost on $\H$ are not symmetric if one
swaps the indices $1$ and $2$ (although this could be overcome by changing also
$\mu$). 
\

\noindent\textsc{Trajectories ---} We are going to define the
trajectories of our optimal control problem by using the approach
via differential inclusions which is rather convenient here. This
approach has been introduced in \cite{Wa} (see also \cite{AF}) and has
now become classical.

Our trajectories
$\trajxt(\cdot)=\big((\trajxt)_1,(\trajxt)_2,\dots,(\trajxt)_N\big)(\cdot)$
are Lipschitz continuous functions which are solutions  of the following
differential inclusion
\begin{equation} \label{def:traj} 
  \dottrajxt (s) \in \mB(\trajxt (s),t-s) \quad \hbox{for a.e.  } s \in
  [0,t)  \: ; \quad \trajxt(0)=x
\end{equation} 
where 
\begin{equation}
  \mB(z,s):= \begin{cases} \mathbf{B}_i(z,s)   &   \text{ if }
      z\in\Omega_i\,,    \\ \cob \big( \mathbf{B}_1(z,s) \cup
      \mathbf{B}_2(z,s) \big)  &  \text{ if } z\in\H\,, 
   \end{cases} 
\end{equation}
the notation $\cob(E)$ referring to the convex closure of the set
$E\subset\R^N$.  We point out that if the definition of $\mB(z,s)$ is
natural when $z\in\Omega_i$, it is dictated by the assumptions to obtain
the existence of a solution to (\ref{def:traj}) for $z\in\H$ (see
below).

As we see, our controls $a(\cdot)$ can take two forms: either $a(s)$
belongs to one of the control sets $A_i$; or it can be expressed as a
triple $(\alpha_1(s),\alpha_2(s),\mu(s))\in A_1\times A_2\times[0,1]$.  Hence, in
order to define globally a control, we introduce the compact set 
\begin{equation*}
    A:= A_1\times A_2\times [0,1]
\end{equation*}
and define a control as
being a function of $L^\infty (0,t; A)$ which can be seen as a subset of $\Ac := L^\infty (0,T; A)$.  Let us define 
\begin{equation*}
    \Ei:= \big\{ s\in(0,t):\trajxt (s) \in \Omega_i  \big\} \,,\quad
    \E_\H:= \big\{s\in(0,t):\trajxt (s) \in \H   \big\} \,,
\end{equation*}
where actually these sets depend on $(x,t)$ but we shall  omit this dependence
for the sake of simplicity of notations. We then have the following 

\begin{theorem}\label{def:dyn}
 Assume \hyp{\Omega}{}, \hyp{\rm C}{1}, \hyp{\rm C}{2} and 
 \hyp{\rm C}{3}. Then
 \begin{itemize}
     \item[{\rm (i)}] For each $x\in \R^N$, $t\in[0,T)$ there exists a
	 Lipschitz function $\trajxt : [0,t]\to \R^N$ which is a
	 solution of the differential inclusion  \eqref{def:traj}.
     \item[{\rm (ii)}] For each solution  $\trajxt(\cdot)$ of
	 \eqref{def:traj},  there exists a control $\a(\cdot)\in \Ac$
	 such that for a.e. $s\in(0,t)$
         \begin{equation}
	     \begin{aligned}\label{fond:traj} 
		 \dottrajxt (s) &= \sum_{i=1,2}b_i\big(\trajxt
		 (s),t-s,\alpha_i(s)\big)\mathds{1}_{\Ei}(s)+
		 b_\H\big(\trajxt(s),t-s,a(s)\big)\mathds{1}_{\E_\H}(s)
             \end{aligned}
        \end{equation} 
	where $a(s)=\big(\alpha_1(s),\alpha_2(s),\mu(s)\big)$ if
	$\trajxt(s)\in\H$.  
    \item[{\rm (iii)}] If
	$\mathbf{e}(\cdot)=\mathbf{n}_1(\cdot)$ or $\mathbf{n}_2(\cdot)$
	we have 
	\begin{equation*}
		b_\H\big(\trajxt(s),t-s,a(s)\big)\cdot
		\mathbf{e}\big(\trajxt(s)\big) = 0 \quad \hbox{for a.e.
		}s\in\E_\H\;.  
	\end{equation*}
	In other words, $a(s)\in A_0(\trajxt(s),t-s)$ for a.e.
	$s\in\E_\H$.
  \end{itemize} 
\end{theorem}

\begin{proof} The proof is done exactly as in \cite{BBC1}, the only minor
  modification consisting in adding the time variable in the vector
  field $b$.
\end{proof}

\noindent\textsc{Regular and Singular dynamics ---} It is worth
remarking that, in Theorem~\ref{def:dyn}, a solution $\trajxt(\cdot) $
can be associated to several controls $a(\cdot)$. So, to properly set
the control problem we introduce the set $\Tc_{x,t}$ of admissible
controlled trajectories starting from $x$, 
\begin{equation*} 
  \Tc_{x,t}:= \big\{ \big(\trajxt(\cdot),\a(\cdot)\big)\in {\rm
	  Lip}(0,t;\R^N) \times \A \mbox{ such that \eqref{fond:traj} is
	  fulfilled and }  \trajxt(0)=x \big\} \,.
\end{equation*} 
  Given $(z,s)\in\H\times[0,t]$, we call \textit{singular} a dynamic
$b_\H(z,s,\a)$ with $a=(\alpha_1, \alpha_2, \mu)\in A_0(z,s)$ when 
\begin{equation*}
    b_1(z,s,\alpha_1)\cdot \mathbf{n}_1(z) < 0\,, \quad
    b_2(z,s,\alpha_2)\cdot \mathbf{n}_2(z)< 0\,.  
\end{equation*}
Conversely, the \textit{regular} dynamics are
those for which the $b_i (z,s,\alpha_i)\cdot \mathbf{n}_i (z)\geq0$
($i=1,2$). 
The set of regular controls is denoted by
\begin{equation*}
    \Aoreg(z,s):=\big\{\a=(\alpha_1,\alpha_2,\mu)  \in A_0(z,s)\: ; \:
	b_i(z,s,\alpha_i) \cdot \mathbf{n}_i(z)  \geq 0,\ i=1,2  \big\}\,,
\end{equation*}
and the regular trajectories are defined as 
\begin{equation*}
     \Treg:= \Big\{ \big(\trajxt(\cdot),\a (\cdot)\big) \in \Ta :
	  \text{ for a.e. } s\in\E_\H,  \: a(s)\in\Aoreg\big(X(s),t-s\big)
	  \Big\}\,.  
\end{equation*}
 Trajectories satisfying for example 
  $b_1(z,s,\alpha_1)\cdot\mathbf{n}_1(z)<0 <
  b_2(z,s,\alpha_2)\cdot\mathbf{n}_2(z)$ are neither called singular nor
  regular since they do not remain on $\H$, they are handled by
  classical arguments.

\noindent \textsc{ The cost functional --} Our aim is  to minimize a
finite horizon cost functional such that we respectively pay $l_{i}$ if
the trajectory is in $\Omega_i$, and $l_\H$ if it is on $\H$. The
final cost is given by $g$.\\

More precisely,  the cost associated to $(\trajxt(\cdot) ,\a)  \in \Ta$
is 
\begin{equation}\label{def:cost}
  J(x,t;(\trajxt,a)):= \int_0^t \ell\big(\trajxt(s),t-s,a(s) \big) \ds + 
  g\big(\trajxt(t)\big)
\end{equation}
where the Lagrangian is given by 
\begin{equation}\label{def:lagrangian}
    \ell\big(\trajxt (s),t-s,a(s)\big)  := \sum_{i=1,2}
	l_i\big(\trajxt(s),t-s,\alpha_i(s)\big) \mathds{1}_{\Ei}(s) +
        l_\H\big(\trajxt(s),t-s,\a(s)\big)\mathds{1}_{\E_\H}(s)\,. 
\end{equation}

\noindent\textsc{The value functions -- }For each $x\in\R^N$ and
$t\in[0,T)$, we define the following two value functions
\begin{align}
    \VFm (x,t)&:= \inf_{(\trajxt,\a) \in \Ta} J\big(x,t; (\trajxt,
    \a)\big) \label{eq:valorem}\\
 \VFp (x,t)&:=
    \inf_{(\trajxt,\a) \in \Treg} J\big(x,t; (\trajxt, \a)\big).\label{eq:valorep}
\end{align}
 
A first key result is the {\bf Dynamic Programming Principle} 
(the proof being standard once we have the definition of trajectories, we skip it).

\begin{theorem} 
  Assume \hyp{\Omega}{}, \hyp{\rm C}{1}, \hyp{\rm C}{2} and \hyp{\rm
      C}{3}. Let $\VFm,\VFp$ be the  value functions defined in
  \eqref{eq:valorem} and \eqref{eq:valorep}.  Then for each
  $(x,t)\in\R^N\times[0,T)$, and each $\tau\in(0,t)$, we have 
  \begin{align} 
    \VFm (x,t)&= \inf_{(\trajxt,\a) \in \Ta} \left\{ \int_0^\tau
	\ell\big(\trajxt (s),t-s,a(s) \big)   \ds +   \VFm (\trajxt(\tau), t-\tau)
    \right\} \label{DPPm} \\
    \VFp (x,t)&= \inf_{  (\trajxt,\a) \in \Treg}\left\{ \int_0^\tau
	\ell\big(\trajxt (s),t-s,a(s) \big)  \ds + \VFp (\trajxt(\tau), t-\tau)
    \right\}.\label{DPPp} 
  \end{align} 
\end{theorem} 

We will prove that both value functions are continuous, but here it is
not so immediate since we only assume controlability in the normal
directions. We postpone this proof which uses some comparison for the
semi-continuous envelopes.

\section{The pde formulation of the problem} \label{sect:edp}

In order to describe what is happening on the hypersurface $\H$, we
shall introduce two "tangential Hamiltonians", namely $\HT,\HTreg$.
We introduce some notations to be clear on how they are defined.

We shall consider the tangent bundle $T\H:=\cup_{z\in\H}\big(\{z\}\times T_z\H\big)$
where $T_z\H$ is the tangent space to $\H$ at $z$ (which is
essentially $\R^{N-1}$). Thus, if $\phi\in C^1(\H)$, and
$x\in\H$, we denote by $D_\H \phi(x)$ the gradient of $\phi$ at
$x$, which belongs to $T_{x}\H$.

Also, the scalar product in $T_z\H$ will be denoted by
$\psg{u}{v}$ (we drop the reference to $T_z\H$ for simplicity, since no
confusion has to be feared in the sequel).
In this definition, both vectors $u,v$ should belong to
$T_z\H$ for this definition to make sense. Hence, to be precise we
should use the orthogonal projection $P_z:\R^N\to T_z\H$ when at
least one of the vectors $u,v$ lives in $\R^N$, but we shall omit this
point when writing $\psg{b_\H(x,t,a)}{D_\H\phi(x,t)}$\,.
Indeed, for any control $a$ in $A_0(x,t)$ or $\Aoregxt$,
$b_\H(x,t,a)$ can be identified with $P_x b_\H(x,t,a)$ since
$b_\H(x,t,a)$ has no component on the normal direction to $\H$,
by definition. To avoid confusions, the notation $u\cdot v$ will refer
only to the usual Euclidian scalar product in $\R^N$. 

The Hamiltonians $\HT,\HTreg$ will be written as $H_T/\HTreg(x,t,p)$ where
$((x,p),t)\in T\H\times[0,T]$. They are defined as follows:
\begin{align}
    \HT(x,t,p):=\sup_{A_0(x,t)} \big\{  - \psg{
    b_\H(x,t,\a)}{p}   -  l_\H(x,t,\a)   \big\}\,, \label{def:HamHT}  \\
    \HTreg(x,t,p):=\sup_{\Aoregxt} \big\{- \psg{
    b_\H(x,t,\a)}{p}-l_\H(x,t,\a)\big\}\,, \label{def:HamHTreg}
\end{align} where $A_0(x,t),\Aoreg(x,t)$ have been defined above.

The definition of viscosity sub and super-solutions for $\HT$ and
$\HTreg$ have to be understood on $\H$ as follows: 
\begin{definition}   \label{defi:sousolH}
    A bounded usc function $u:\H\times[0,T]\to\R$ is a viscosity
    subsolution of
$$
	u_t(x,t)+\HT(x,t,D_\H u)=0\quad \text{on}\quad
	\H\times[0,T]
$$
if, for any $\phi\in C^1(\H\times[0,T])$ and any maximum point
    $(x,t)$ of $(z,s)\mapsto u(z,s)-\phi(z,s)$ in $\H\times[0,T]$,
    one has
    \begin{equation*}
	\phi_t(x,t)+\HT\big(x,t,D_\H\phi(x,t)\big)\leq0\;.
    \end{equation*}
\end{definition}

Notice that of course, $(x,D_\H \phi(x,t))\in T\H$, so that this
is coherent with the definition of $\HT$.  A similar definition
holds for $\HTreg$, for supersolutions and solutions.  Of course, if $u$
is defined in a bigger set containing $\H\times[0,T]$ (typically $\R^N\times[0,T]$), we
have to use $u|_{\H\times[0,T]}$ (the restriction of $u$ to $\H\times[0,T]$)
in this definition, a notation that we will omit when not
necessary.

For the sake of clarity we introduce now a global formulation involving
a complementary Hamiltonian on the interface $\H$. To begin with, we recall that a subsolution (resp. a supersolution) of \eqref{eq:gen} when $H(x,t,p)=H_1(x,t,p)$ if $x \in  \Omega_1$ and $H(x,t,p)=H_2(x,t,p)$ if $x \in  \Omega_2$ is a bounded usc function $u$ (resp.a bounded lsc function $v$) which satisfies
\begin{align}
 & \begin{cases}
    u_t + H_1(x,t,Du) \leq  0 & \mbox{ in }  \Omega_1 \times (0,T)\,,\\
    u_t + H_2(x,t,Du) \leq  0 & \mbox{ in }  \Omega_2 \times (0,T)\,, \\ 
   u_t + \min\{H_1(x,t,Du), H_2(x,t,Du) \} \leq  0  & \mbox{ in }  \H \times (0,T) \,,\\
\end{cases}\label{eqn:ss-solH} \\[2mm]
\Bigg[\quad\text{resp.}\quad &
\begin{cases}
    v_t + H_1(x,t,Dv) \geq  0 & \mbox{ in }  \Omega_1 \times (0,T)\,,\\
    v_t + H_2(x,t,Dv) \geq  0 & \mbox{ in }  \Omega_2 \times (0,T)\,, \\ 
   v_t + \max\{H_1(x,t,Dv), H_2(x,t,Dv) \} \geq  0  & \mbox{ in }  \H \times (0,T)\\
\end{cases} \quad\Bigg]\;.\label{eqn:sur-solH}
\end{align}

Recall that since each $b_{i}$ is defined on 
$\overline{\Omega_{i}}\times(0,T)\times\R$, then 
$H_{i}$ is well-defined on $\H\times(0,T)$.
Next we have the following definition.
\begin{definition}
\label{solviscogen}
We say that a bounded usc function $u$ is a 
\underline{{\rm subsolution}} of  
\begin{align}
 u_t + \HHm(x,t, Du) & = 0  \hbox{ in }  \R^N  \times (0,T)\label{HHm-eqn}\\[2mm]
\big[\quad\text{resp.}\quad
 u_t + \HHp (x,t, Du) & = 0  \hbox{ in }  \R^N  \times (0,T)\label{HHp-eqn}
\qquad\big]
\end{align}
if it satisfies \eqref{eqn:ss-solH} and
\begin{align*}
	u_t(x,t)+\HT(x,t,D_\H u) &\leq 0\quad \text{on}\quad
	\H\times[0,T]\;,\\[2mm]
	\big[\quad\text{resp.}\quad 
	u_t(x,t)+\HTreg (x,t,D_\H u) &\leq 0\quad \text{on}\quad
	\H\times[0,T]\;, \quad\big]
\end{align*}	

in the sense of Definition~\ref{defi:sousolH}.\\
A lsc function $v$ is a  \underline{{\rm supersolution}} of \eqref{HHm-eqn} or \eqref{HHp-eqn} if it satisfies \eqref{eqn:sur-solH}.
\end{definition}

Notice that in this definition, a complementary condition is required only for 
the subsolution, nothing more is added for the supersolution.

\subsection{Properties of $\VFp$ and $\VFm$}\label{pvpvm}

We shall prove later on that both $\VFp$ and $\VFm$ are continuous,
but for the moment we
have to treat them a priori as discontinuous viscosity solutions of some
problem. We recall that, for any bounded function $v$, the lower and
upper semi-continuous envelopes are defined by
$$v_*(x,t):=\liminf_{(z,s)\to(x,t)}v(z,s)\,,\quad
v^*(x,t):=\limsup_{(z,s)\to(x,t)}v(z,s)\,.$$ 
Then, as we mention in the introduction the
definition of viscosity solution for discontinuous solutions is modified
by taking $(\VFm)_*$ instead of $\VFm$ for the supersolution condition,
and $(\VFm)^*$ instead of $(\VFm)$ for the subsolution condition.

We claim that the value functions  $\VFm$ and  $\VFp$ are viscosity
solutions of the Hamilton-Jacobi-Bellman problem
(\ref{Bellman-Om})-(\ref{Bellman-H-sub})-(\ref{Bellman-H-sup}), while
they fulfill different  inequalities on the hyperplane $\H$.

\begin{theorem}\label{teo:HJ}
    Assume \hyp{g}{}, \hyp{\Omega}{} and \hyp{\rm C}{}.  Then value functions
    $\VFm$ and $\VFp$ are both viscosity solutions of $u_t+H(x,u,Du)=0$.
    Moreover, $\VFm$ is a subsolution of $u_t+\HHm(x,t,Du)=0$ while
    $\VFp$ is a subsolution of $u_t+\HHp(x,t,Du)=0$.
\end{theorem}

\begin{proof} 
 The proof follows the arguments of \cite[Thm 2.5]{BBC1} with some adaptations due to the fact that $\VFm, \VFp$ can be discontinuous. We briefly show how to adapt the arguments.  In order to prove that $(\VFm)_*$ is a
supersolution we consider a point $(x,t)$ where $(\VFm)_*-\phi$ reaches
its minimum, $\phi$ being a smooth test function. If $x$ belongs to
some $\Omega_i$, the proof is classical since everything can be done in
$\Omega_i$ around the time $t$.
	
Thus we assume that $x\in\H$ and that the minimum is strict in
$B(x,r)\times(t-\sigma,t+\sigma)$ for some $r,\sigma>0$.
There exists a sequence $(x_n,t_n)\in B(x,r)\times(t-\sigma,t+\sigma)$ 
which converges to
$(x,t)$ such that $\VFm(x_n,t_n) \to (\VFm)_*(x,t)$ and by the dynamic programming principle,
\begin{equation*}
    \VFm(x_n,t_n)= \inf_{(\trajxntn,a)\in\Tan}\Big\{
	\int_{0}^{\tau}\ell\big(\trajxntn(s),t_n- s,a(s)\big)\ds +
	\VFm \big( \trajxntn(\tau), t_n-\tau\big)\Big\}\,, 
\end{equation*}
where $\tau<\sigma$.  Using that (i) $\VFm(x_n,t_n) = (\VFm)_*(x,t)+o_n(1)$ where $o_n(1) \to 0$, (ii) $\VFm \big( \trajxntn(\tau), t_n-\tau\big) \geq \VFm_* \big( \trajxntn(\tau), t_n-\tau\big)$ and the maximum point property, we obtain
\begin{equation*}
    \phi (x_n,t_n)+o_n(1) \geq \inf_{(\trajxntn,a)\in\Tan}\Big\{
	\int_{0}^{\tau}\ell\big(\trajxntn(s),t_n- s,a(s)\big)\ds +
	\phi\big( \trajxntn(\tau), t_n-\tau\big)\Big\}\,.
\end{equation*}
Now we use  the expansion of $\phi(\trajxntn(\tau),t_n-\tau)$, and
noting $X(\cdot)=\trajxntn(\cdot)$ for simplicity, we rewrite
the inequality as $o_n(1) \leq \sup_{(X,\a) } \int_{0}^{\tau}
\delta[\phi](s) \ds$ where
\begin{equation*}
    \begin{aligned} & \delta[\phi](s):=\\
	& \Big(-l_1(X(s),t_n-s,\alpha_1(s)) -
	    b_1(X(s),t_n-s,\alpha_1(s))\cdot D\phi(X(s),t_n-s)+
	    \phi_t(X(s),t_n-s)\Big)\mathds{1}_{\E_1}(s)\\ & +
	    \Big(-l_2(X(s),t_n-s,\alpha_2(s)) - b_2(X(s),t_n-s,\alpha_2(s))\cdot
	    D\phi(X(s),t_n-s) +\phi_t(X(s),t_n-s) \Big)\mathds{1}_{\E_2}(s)\\ & +
	    \Big(-l_\H(X(s),t_n-s,\a(s)) - b_\H(X(s),t_n-s,\a(s))\cdot
	    D\phi(X(s),t_n-s)+\phi_t(X(s),t_n-s)\Big) \mathds{1}_{\E_H}(s)
	    \\ & \leq 	
	    \bigg(\phi_t(X(s),t_n-s)+H_1\Big(X(s),t_n-s,D\phi(X(s),t_n-s)\Big)\bigg)
	    \mathds{1}_{\E_1}(s)\\ & + \bigg(\phi_t(X(s),t_n-s)+
	    H_2\Big(X(s),t_n-s,D\phi(X(s),t_n-s)\Big)\bigg)
	    \mathds{1}_{\E_2}(s)\\ &+ \bigg(\phi_t(X(s),t_n-s)+
	    \HT\Big(X(s),t_n-s,D\phi(X(s),t_n-s)\Big)\bigg)
	    \mathds{1}_{\E_H}(s)\,.
        \end{aligned} 
\end{equation*}

Using that $H_1,H_2,\HT\leq\max(H_1,H_2)$ (only on $\H$ for $\HT$), letting $n\to\infty$ and then
dividing by $\tau$ and sending $\tau$ to zero, we
obtain
\begin{equation*}
    \max\big(\phi_t+H_1,\phi_t+H_2\big)\big(x,t,D\phi(x,t)\big)\geq0\,,
\end{equation*}
which is the viscosity supersolution condition. The proof for $(\VFp)_*$
is exactly the same, with $\HT$ replaced by $\HTreg$, which satisfies
also $\HTreg\leq\max(H_1,H_2)$ on $\H$.
    
 For the subsolution condition, we have to consider maximum points of $(\VFm)^*-\phi$, $\phi$ being again a smooth function. If such maximum point are in $\Omega_1$ or $\Omega_2$, the proof is again classical. Hence we consider the case when $(\VFm)^*-\phi$ reaches a strict local maximum at $(x,t)$ with $x\in \H$, $t\in (0,T)$.

Then there exists a sequence $(x_n,t_n)\to(x,t)$ such that $\VFm(x_n,t_n)\to (\VFm)^* (x,t)$ and our first claim is that we can assume that $x_n \in \H$. Indeed, if $x_n \in \Omega_1$, we use assumption \hyp{C}{4} : there exists $\alpha_i$ such that $ b_1(x,t,\alpha_1) \cdot \mathbf{n}_1(x)=\delta$. Considering the trajectory with the constant control $\alpha_1$ 
$$ \dot Y(s) = b_1(Y(s),t_n-s,\alpha_1)\quad,\quad Y(0)=x_n ,$$
it is easy to show that $\tau^1_n$, the first exit time of the trajectory $Y$ from $\Omega_1$ tends to $0$ as $n\to +\infty$. By the Dynamic Programming Principle, denoting $(\tilde x_n,\tilde t_n)= (X(\tau^1_n),t-\tau^1_n)$, we have
$$ \VFm(x_n,t_n)\leq 
	\int_{0}^{\tau^1_n}\ell\big(Y(s),t_n- s,\alpha_1\big)\ds +
	\VFm (\tilde x_n,\tilde t_n)= \VFm (\tilde x_n,\tilde t_n) + o_n(1),$$
where $o_n(1)\to 0$. Therefore $\VFm (\tilde x_n,\tilde t_n) \to (\VFm)^* (x,t)$ and $\tilde x_n \in \H$.

Assuming that $x_n \in \H$, we can use again the Dynamic Programming Principle
$$ \VFm(x_n,t_n)\leq 
	\int_{0}^{\tau}\ell\big(\trajxntn (s),t_n- s,a(s)\big)\ds +
	\VFm ( \trajxntn(\tau), t_n-\tau\big),
	$$
with constant controls $a(s)=\alpha_i$ with $b_i (x,t,\alpha_i)\cdot \mathbf{n}_i (x) <0$. Arguing as above we get
$$ \phi_t (x,t) -b_i(x,t,\alpha_i)\cdot
D\phi(x,t)-l_i(x,t,\alpha_i)\leq0 \; .$$
Moreover, combining Assumptions \hyp{\rm C}{3} and \hyp{\rm C}{4}, one proves easily that this inequality holds
for any $\alpha_i$ with $b_i(x,t,\alpha_i)\cdot \mathbf{n}_i (x) \leq 0$.

Taking these informations into account, if we assume by contradiction that
\begin{equation*}
    \min\big\{\phi_t (x,t) +H_1\big(x,t,D\phi(x,t)\big)\,;
    \,\phi_t (x,t)+H_2\big(x,t,D\phi(x,t)\big)\big\} > 0\,,
\end{equation*}
this means that there exists $\alpha_1,\alpha_2$ with if $b_1(x,t,\alpha_1)\cdot\mathbf{n}_1(x)>0$ and
$b_2(x,t,\alpha_2)\cdot\mathbf{n}_2(x)>0$ such that, for $i=1,2$
$$ \phi_t (x,t) -b_i(x,t,\alpha_i)\cdot
D\phi(x,t)-l_i(x,t,\alpha_i) > 0 \; .$$

For $(y,s)$ close to $(x,t)$ and for such $\alpha_1,\alpha_2$, we set
\begin{equation*}
    \mu^\sharp(y,s):= \frac{b_2(y,s,\alpha_2) \cdot
	\mathbf{n}_2(y)}{b_1(y,s,\alpha_1)\cdot\mathbf{n}_1(y)+
	b_2(y,s,\alpha_2)\cdot \mathbf{n}_2(y)}\,.
\end{equation*}
Then we solve the ode 
\begin{equation*}
    \dot{x}(s)=\mu^\sharp(x(s),t-s)b_1(x(s),t-s,\alpha_1)+
    (1-\mu^\sharp(x(s),t-s))b_2(x(s),t-s,\alpha_2)\, , \quad \quad x(0)=x. 
\end{equation*}
By our hypotheses on $b_1$ and $b_2$, the right-hand side is Lipschitz
continuous so that the Cauchy-Lipschitz applies and gives a solution
$x(s)$.  Moreover, by our choice of $\mu^\sharp$, it is clear that
$0\leq\mu^\sharp\leq 1$ and that $\dot{x}(s)\cdot\mathbf{n}_1(x(s))=0$,
which implies by Gronwall's lemma that $s\mapsto x(s)$ remains on
$\H$, at least until some time $\tau>0$.  Using again the Dynamic Programming Principle and the usual arguments, we are lead to
$$\begin{aligned}
& \mu^\sharp(x,t)\biggl(\phi_t (x,t) -b_1(x,t,\alpha_1)\cdot
D\phi(x,t)-l_1(x,t,\alpha_1)\biggr)\\
& +(1-\mu^\sharp(x,t))\biggl(\phi_t (x,t) -b_2(x,t,\alpha_2)\cdot
D\phi(x,t)-l_2(x,t,\alpha_2)\biggr)\leq 0\;,\end{aligned}$$
a contradiction. 

Finally the $\HT$-inequality follows from the same arguments : in particular, if $b_1(x,t,\alpha_1)\cdot\mathbf{n}_1(x)<0$ and
$b_2(x,t,\alpha_1)\cdot\mathbf{n}_2(x)<0$, the above $\mu^\sharp$-argument can be applied readily.

The same proof works also for $(\VFp)^*$, except that some situation
cannot occur since we are only considering regular dynamics.
\end{proof}

Our next result is a (little bit unusual) supersolution property which
is satisfied by $ \VFp$ on $\H$, which is done exactly as in of
\cite[Thm 2.7]{BBC1} once we have the following extension result

\begin{lemma}\label{lem:extension}Let us assume that \hyp{\Omega}{} holds and let
  $\phi\in C^1\big(\H\times[0,T]\big)$.  Then there exists a function
  $\tilde \phi\in C^1\big(\R^N\times[0,T]\big)$ such that
  $\tilde\phi=\phi$ in $\H\times[0,T]$.  
\end{lemma}

\begin{proof} The proof is rather classical so that we omit it.
%
%
\end{proof}

We are going to consider control problems set in either $\Omega_i$ or
its closure.  For the sake of clarity we use the following notation.  If
$x \in \Omega_i$, and $\alpha_i (\cdot)\in L^\infty([0,T];A_i)$, we will
denote by $\trajyali(\cdot)$ the solution of the following ode
\begin{equation}\label{def:trajY}
    \dottrajyali(s) = b_i(\trajyali(s), t-s, \alpha_i(s))\quad,\quad
    \trajyali(0) = x\: .  
\end{equation}

 The following result is playing a key role in order to
  prove that the value function $\VFp$ is continuous and the maximal
  subsolution of
  \eqref{Bellman-Om}-\eqref{Bellman-H-sub}-\eqref{Bellman-H-sup}-\eqref{eqnid}
  (see Theorem~\ref{thm:Um.Up} below). One of the key difference between
  the $\VFm$ and $\VFp$ cases is that for the
  $\VFp/\HTreg$ case,
  we are able to prove such result only for the supersolution
  $(\VFp)_*$, while, in the other case ($\VFm/\HT$), it is true for
  any supersolution (see Theorem~\ref{teo:condplus} below).
  
\begin{theorem}\label{teo:condplus.VFp}
    Assume \hyp{g}{}, \hyp{\Omega}{} and \hyp{\rm C}{}.  Let $\phi\in
    C^1\big(\H\times[0,T]\big)$ and suppose that $(x,t)$ is a
    minimum point of $(z,s) \mapsto (\VFp)_* (z,s)-\phi (z,s)$ in
    $\H\times[0,T]$.  Then we have either  \\[2mm] {\bf A)} there
	exist $\eta >0$, $i\in\{1,2\}$ and a control $\alpha_i (\cdot)$
	such that, $\trajyali(s) \in \Omega_i $ for all $s \in ]0,\eta]$
	and
	\begin{equation}
	    (\VFp)_*(x,t) \geq \int_0^{\eta} l_i(\trajyali(s),t-s,
	    \alpha_i(s)) \ds + (\VFp)_* (\trajyali(\eta),t-\eta)
	\end{equation} or \\ {\bf B)} it holds
       	\begin{equation}\label{condhyperSUP.VFp}
	    \partial_t \phi(x,t)+\HTreg \big(x,t, D_\H\phi(x,t)\big) \geq
	    0.
	\end{equation}
\end{theorem}

\begin{proof}
Since $x\in\H$, by assumption \hyp{\rm C}{3}, there exists a regular
optimal control $a(\cdot) \in\Treg$ such that
\begin{equation*}
    \VFp(x,t)= \int_0^t \ell\big(\trajxt(s),t-s,a(s)
    \big)\ds+g(\trajxt(t))\;. 
\end{equation*}
Moreover, by the Dynamic Programming Principle, we have, for any
$\tau>0$ 
\begin{equation*}
    \VFp(x,t)= \int_0^\tau \ell\big(\trajxt(s),t-s,a(s) \big) \ds +
    \VFp(\trajxt(\tau),t-\tau)\;.
\end{equation*}
We argue depending on whether or not there exists a sequence
$(\tau_k)_k$ converging to $0$ such that $\tau_k >0$ and
$\trajxt(\tau_k) \in \H$.

If it is NOT the case then this means that we are in the case {\bf A)}
since, for $\eta$ small enough, the trajectory  $\trajxt(\cdot)$ stays
necessarily either in $\Omega_1$ or in $\Omega_2$ on $]0,\eta]$.
Therefore we can assume for instance that
$\trajxt(\cdot)=\trajyali(\cdot)$ and take $\tau=\eta$ in the above
equality.

On the contrary, if IT IS the case, we can use the minimum point
property: assuming without loss of generality that $\phi(x,t)=(\VFp)_* (x,t)$, we extend $\phi$ to $\R^N\times[0,T]$ thanks to
Lemma~\ref{lem:extension} and write, for $k$ large enough,
\begin{equation*}
    \tilde\phi(x,t)\geq  \int_0^{\tau_k} \ell\big(\trajxt(s),t-s,a(s)
    \big)\ds + \tilde\phi(\trajxt(\tau_k),t-\tau_k)\;.
\end{equation*}
The rest of the proof is the same as \cite[Thm 2.7]{BBC1}: we obtain a
contradiction by assuming $$\phi_t(x,t)+\HTreg \big(x,t,D_\H\phi(x,t)\big) \leq -\eta < 0\,,$$ using the normal
controllability condition \hyp{\rm C}{4} instead of the more general
(and usual) one which was used in \cite{BBC1}.
\end{proof}

\subsection{Properties of sub and supersolutions}\label{pss}
\ 

\begin{theorem} \label{teo:sotto}
    Assume \hyp{\Omega}{} and \hyp{\rm C}{}. If $u:\R ^N\times[0,T] \to
  \R$ is a bounded viscosity subsolution of $u_t+H(x,t,Du)=0$, then
  $u$ is a subsolution of $u_t+\HHp(x,t,Du)=0$.
\end{theorem}

\begin{proof} It is enough to check the subsolution condition only on $\H$
   since the property clearly holds in each $\Omega_i$ by definition.

    We recall that $u^*|_{\H\times[0,T]}$ is the restriction of $u^*$ to
  $\H\times[0,T]$. Let $\phi(\cdot)$ be a $C^1$-function on $\H$
  and $(\bar{x},\bar{t})$ a maximum point of $u^*|_{\H\times[0,T]}-\phi$ on
  $\H\times[0,T]$. Our aim is  then to prove that, for any $\a\in
  \A_0^{\rm reg}(\bar{x},\bar{t})$ we have
  \begin{equation}\label{ineq:sub.gen}
      \phi_t(\bar{x},\bar{t}) -\psg{b_\H\big(\bar{x},\bar{t},a\big)}{D_\H\phi(\bar{x},\bar{t})}
      -l_\H\big(\bar{x},\bar{t},a\big) \leq 0\,.
  \end{equation}
  This proof follows \cite[Thm. 3.1]{BBC1} so that we only mention here
  the modifications. First, we extend $\phi$ by $\tilde\phi$ given by
  Lemma~\ref{lem:extension}. Then for
  $\varepsilon\ll1$ and $(z,s)\in \H\times[0,T]$ we consider the
  function
  \begin{equation}
      (z,s)\mapsto u(z,s)-\tilde\phi(z,s)-\eta\,\dH(z)-
      \frac{\dH(z)^2}{\varepsilon^2}-|z-x|^2-|s-t|:=
      u(z,s)-\psi_\varepsilon(z,s)\,,
  \end{equation}
  where we recall that $\dH(\cdot)$ is the signed distance function to $\H$
  which is positive in $\Omega_1$ and negative in
  $\Omega_2$.
    
  Writing $a=(\alpha_1,\alpha_2,\mu)$, we assume that we are in the
  situation when $b_1(\bar x,\bar t,\alpha_1)\cdot\mathbf{n}_1(\bar
  x)<0$ (and the same for index 2), since the case of
  non-strict inequalities can be recovered by hypothesis \hyp{\rm C}{4}
  as in Thm.~\ref{teo:HJ} (recall that $a$ being a regular control, the
  opposite signs are forbidden).  We choose $\eta>\bar\eta$ where $\bar\eta$
  is a solution of the following equation (which has a solution under
  the assumption above of strict signs):
  \begin{equation*}
      \tilde\phi_t(\bar x,\bar t) -b_1(\bar x,\bar t,
      \alpha_1)\cdot\big( D\tilde\phi(\bar x,\bar t) +\bar \eta
      \mathbf{n}_2(\bar x)\big) -l_1(\bar x,\bar t,\alpha_1)=0\,.
  \end{equation*}
  The rest of the proof follows the cited reference: thanks to the
  penalization terms, for $\eps$ small enough, $u^*-\psi_\varepsilon$
  reaches its max at some point
  $(x_\eps,t_\eps)\in\overline{\Omega_2}\times[0,T]$.  
  Then, using the equation in $\Omega_2\times[0,T]$ or on $\H\times[0,T]$ leads to
  \begin{equation*}
      \tilde\phi_t(\bar x,\bar t) -b_2(\bar x,\bar
      t,\alpha_2)\cdot\big(D\tilde\phi(\bar x, \bar
      t)+\eta\mathbf{n}_2(\bar x)\big)-l_2(\bar x,\bar t,\alpha_2)\leq
      o_\eps(1)\,.
  \end{equation*}
  We let $\eps$ tend to zero first, and then $\eta$ to $\bar\eta$. Using
  the specific value of $\bar \eta$ leads to
  \begin{equation*}
      \tilde\phi_t(\bar x, \bar t)-b_\H(\bar x,\bar t,a)\cdot
      D\tilde\phi(\bar x,\bar t)-l_\H(\bar x,\bar t,a)\leq0\,,
  \end{equation*}
  that we interpret as \eqref{ineq:sub.gen} since $b_\H(\bar x,\bar
  t,a)$ has no component on the normal direction to $\H$ and by
  construction, $D_\H(\tilde\phi|_\H)=D_\H\phi$.
\end{proof}

The following lemma states a super and a sub optimality principle respectively for super and subsolutions of $w_t+H(x,t,Dw)=0$. The proof is classical (see \cite{BP2,Bl1,Bl2} and also the proof of  \cite[Lem. 3.2]{BBC1}).
\begin{lemma}\label{reverse}
    Assume \hyp{\Omega}{} and \hyp{\rm C}{}. Let $v:\R^N\times[0,T] \to \R$ be a lsc
    supersolution of $v_t+H(x,t,Dv)=0$ and 
    $u:\R^N\times[0,T] \to \R$ be a usc subsolution of
    $u_t+H(x,t,Du)=0$.
    Then, if $x \in \Omega_i$ ($i\in\{1,2\}$), we have for all $\sigma\in[0,t]$ 
    \begin{equation}\label{reverse:super} 
	v(x,t) \geq \inf_{\alpha_i(\cdot),\theta_i}
	\,\biggl[\int_0^{\sigma\wedge\theta_i}
	l_i\big(\trajyali(s),t-s,\alpha_i(s)\big) \ds
	+v\big(\trajyali(\sigma\wedge\theta_i),t-(\sigma\wedge\theta_i)\big)\;\biggr]\; ,
    \end{equation}
    and 
    \begin{equation}\label{reverse:sub} 
	u(x,t) \leq \inf_{\alpha_i(\cdot)}\sup_{\theta_i}
	\,\biggl[\int_0^{\sigma\wedge\theta_i}
	l_i\big(\trajyali(s),t-s,\alpha_i(s)\big) \ds
	+u(\trajyali(\sigma\wedge\theta_i),t-(\sigma\wedge\theta_i)\big)\;\biggr]\;,
    \end{equation}
    where $\trajyali$ is the solution of the ode \eqref{def:trajY} and
    the infimum/supremum is taken on all stopping times $\theta_i$ such that
    $\trajyali(\theta_i) \in \partial\Omega_i$ and $\tau_i \leq \theta_i \leq \bar
    \tau_i$ where $\tau_i$ is the first exit time of the trajectory
    $\trajyali$ from $\Omega_i$ and $ \bar \tau_i$ is the one from
    $\overline{\Omega}_i $. 
\end{lemma}

The following important result highlights the fundamental
alternative: given $x\in\H$, either
there exists an optimal strategy consisting in entering in $\Omega_1$ or
$\Omega_2$, or all the optimal strategies consist in staying on $\H$
at least for a while.

\begin{theorem}\label{teo:condplus}
    Assume \hyp{\Omega}{} and \hyp{\rm C}{}. Let $v:\R^N\times[0,T] \to
    \R$ be a lsc supersolution of $v_t+H(x,t,Dv)=0$. Let
    $\phi\in C^1\big(\H\times[0,T]\big)$ and $(x,t)$ be a minimum
    point of $(z,s)\mapsto v(z,s)-\phi(z,s)$.  Then, the following
    alternative holds:
	\begin{description}
	    \item[\textbf{A)}] either there exist $\eta>0$, $i\in\{1,2\}$
	        and a sequence $x_k \in \overline{\Omega}_i $
		converging to $x$
		such that $v(x_k,t) \to v(x,t)$ and,
		for each $k$, there exists  a control
		$\alpha_i^k(\cdot)$ such that the corresponding
		trajectory $\trajyxki(s) \in \overline{\Omega}_i $ for
		all $s \in [0,\eta]$ and
		\begin{equation}\label{condA} 
		    v(x_k,t) \geq \int_0^{\eta} l_i\big(\trajyxki(s),t-s,
		    \alpha^k_i(s)\big) \ds +
		    v\big(\trajyxki(\eta),t-\eta\big) \;;
		\end{equation}
	    \item[\textbf{B)}] or there holds
		\begin{equation}
		    \label{condhyperSUP}
		    \phi_t(x,t)+\HT\big(x,t,D_\H\phi(x,t)\big)\geq0.
		\end{equation}
	\end{description}
\end{theorem}

\begin{proof}
    As in \cite[Thm. 3.3]{BBC1}, we are going to prove that if
    \textbf{A)} does not hold, then necessarily the second possibility
    holds. Up to a standard modification of $\phi$, we may assume that
    the max is strict. For $\eps>0$ we consider the function
    \begin{equation*}
	v(z,s)-\tilde\phi(z,s)-\delta\dH(z)+\frac{\dH(z)^2}{\eps^2}\;,
    \end{equation*}
 where we recall that $\dH(\cdot)$ is the signed distance function to $\H$ as in
    the proof of Theorem~\ref{teo:sotto}.
    
    There are two cases: either for $\eps$ small enough, the minimum
    point $(x_\eps,t_\eps)$ lies on $\H\times[0,T]$ and this leads
    directly to~\eqref{condhyperSUP} as in \cite[Thm. 3.3]{BBC1}; or we
    may assume that for instance, $x_\eps\in\Omega_i$ for $\eps$ small
    enough. In this second case, the argument by contradiction in
    \cite[Thm 3.3. - 2nd case]{BBC1} applies, using Lemma~\ref{reverse}.   
\end{proof}

\section{Uniqueness result}
\label{sect:uniqueness}

We first prove a local comparison result which is based on auxiliary results in the appendix. To this end, we denote by
$\Q{x_0,t_0}(r,h)$ the open cylinder
$\Q{x_0,t_0}(r,h):=B(x_0,r)\times(t_0-h,t_0)$ where $0<t_0-h<t_0<T$, whose parabolic boundary is given by $$\partial_p
\Q{x_0,t_0}(r,h):=B(x_0,r)\times\{t_0-h\}\cup \partial B(x_0,r)\times [t_0-h,t_0)\,.$$

In the sequel, we assume that $x_0 \in \H$ and that, thanks to \hyp{\Omega}{}, $r$ is small enough in order that there exists a $W^{2,\infty}$-diffeomorphism
    $\Psi=\Psi_{(x_0,r)}$ such that by setting $\tilde
    \Omega:=\Psi\big(B(x_0,r))\big)$, we have
    $$\Psi\big(\H\cap B(x_0,r)\big)=\{x_N=0\}\cap\tilde\Omega\,.$$
We denote this assumption by \hyp{\Omega}{x_0}.

\begin{theorem}\label{thm:main.comp}Assume \hyp{\Omega}{x_0} and
    \hyp{\rm C}{}. If $u$ and $v$ are respectively a bounded usc
    subsolution and a bounded lsc supersolution of $w_t+\HHm(x,t,Dw)=0$
    in $\Q{x_0,t_0}(r,h)$ . Then 
    \begin{equation}
	\Vert(u-v)_+\Vert_{L^\infty(\Q{x_0,t_0}(r,h))} \leq
	\Vert(u-v)_+\Vert_{L^\infty(\partial_p \Q{x_0,t_0}(r,h))}  .
    \end{equation}
\end{theorem}

\begin{proof} We make the change of variable : $\tilde u(x,t):=u(\Psi^{-1}(x),t\big)$, $\tilde
    v(x,t):=v\big(\Psi^{-1}(x),t\big)$. The functions $\tilde u, \tilde
    v$ are respectively sub and supersolution of \eqref{eqn-flat} with $\tilde Q=\tilde\Omega \times (t_0-h,t_0)$, for an Hamiltonian $\tHHm$ associated to 
 $$\tilde b_i(x,t,\cdot):=D\Psi(\Psi^{-1}(x))b_i\big(\Psi^{-1}(x),t,\cdot\big)\,,\
    \tilde l_i(x,t,\cdot):=l_i\big(\Psi^{-1}(x),t,\cdot\big)\quad\text{for }x\in\tilde\Omega,\ t\in [t_0-h,t_0]\,.$$
These dynamics and costs satisfy \hyp{\rm C}{} for some new constants denoted by $\tilde M_b,\tilde L_b, \tilde M_l, \tilde m_l, \tilde \delta$.

We apply Lemma~\ref{lem:flat.comp} which gives \eqref{ineq:flat.comp} which is exactly the result we want by making the change back.
\end{proof}

We now turn to one of our main results, which is the
\begin{theorem}\label{thm:comp.global}
    Assume \hyp{\Omega}{} and \hyp{\rm C}{}.  If $u$ is a bounded, usc subsolution of (\ref{HHm-eqn}) and $v$ is a bounded, lsc supersolution of (\ref{HHm-eqn}), satisfying $u(x,0) \leq v(x,0)$ in $\R^N$, then
    $u \leq v $ in $\R^N \times (0,T)$.
\end{theorem}

\begin{proof} We first prove the 

\begin{lemma}\label{lem:comp.global.subsol}
    For $K>0$ large enough, $\psi(x,t):=-Kt-(1+|x|^2)^{1/2}$
    satisfies $\psi_t+\HHm(x,t,D\psi)\leq-1$ in $\R^N \times (0,T)$.
\end{lemma}

\begin{proof} 
    We just estimate as follows:
$$\psi_t+\HHm(x,t,D\psi)\leq -K+M_b|D\psi|+M_l\leq -K+M_b+M_l \; .$$
    Hence taking $K\geq M_b+M_l+1$ yields the result.
\end{proof}

Using the function $\psi$ of Lemma~\ref{lem:comp.global.subsol}, we introduce, for $\mu\in(0,1)$ close to $1$, the function $u_\mu(x,t):=\mu
    u(x,t)+(1-\mu)\psi(x,t)$. Because of the convexity properties of $H_1,H_2,H_T$, it satisfies
    $(u_\mu)_t+\HHm(x,t,Du_\mu)\leq-(1-\mu)$.  Then we consider
    $$M_\mu:=\sup_{\R^N\times[0,T]}\big(u_\mu(x,t)-v(x,t)\big)\,.$$ 
Since $u_\mu(x,t)\to- \infty$ as $|x|\to\infty$ (uniformly with respect to $t\in[0,T]$) and $v$ is bounded, this ``sup'' is actually a ``max'' and it is achieved at $(x_0,t_0)$. Notice also that $M_\mu \to M:=\sup_{\R^N\times[0,T]}\big(u(x,t)-v(x,t)\big))$ as $\mu \to 1$. We argue by contradiction, assuming that $M>0$, which implies that $M_\mu >0$ for $\mu$ close enough to $1$. From now on, we assume that we have chosen such a $\mu$ and therefore $M_\mu >0$.

Next we remark that $t_0 >0$ since $u_\mu(x,0)-v(x,0) \leq 0$ in $\R^N$ and we first treat the case when $x_0\in \H$. In that way, since \hyp{\Omega}{} holds, we can choose $r>0$, small enough in order that \hyp{\Omega}{x_0} holds. On the other hand, we choose any $h$ such that $t_0-h\geq 0$, say $h=t_0$.
    
The next step consists in introducing the function 
    $$\bar u_\mu(x,t):=u_\mu(x,t)+(1-\mu)^2\big(t-t_0-|x-x_0|^2\big)\; .$$
We claim that $\bar u_\mu$ is a subsolution of $(\bar u_\mu)_t+\HHm(x,t,D\bar
    u_\mu)=0$ for $\mu$ close enough to 1. Indeed, a direct computation gives
\begin{align*}
(\bar u_\mu)_t+\HHm(x,\bar u_\mu,D\bar u_\mu) & \leq 
    (u_\mu)_t+\HHm(x,u_\mu,D u_\mu)+(1-\mu)^2\{1+2M_b r\}\\
    & \leq  -(1-\mu) + (1-\mu)^2\{1+2M_b r\} \leq 0
\end{align*}
for $\mu$ sufficiently close to 1.

    Thus, we use Theorem~\ref{thm:main.comp} with the pair of
    sub/supersolution $(\bar u_\mu,v)$ and we  obtain in particular $$
    M_\mu=u_\mu(x_0,t_0)-v(x_0,t_0)=\bar u_\mu(x_0,t_0)-v(x_0,t_0)\leq
    \Vert(\bar u_\mu-v)_+\Vert_{L^\infty(\partial_p
	\Q{x_0,t_0}(r,h))}\,.$$ 
    However, on the parabolic boundary
    $(\bar u_\mu-v)< M_\mu$. Indeed, on $\partial B(x,r)\times(t_0-h,t_0)$, we have
  $$  \bar u_\mu(x,t)-v(x,t) = u_\mu(x,t)-v(x,t)+(1-\mu)^2\big(t-t_0-r^2\big)\leq M_\mu -(1-\mu)^2r^2\; ,$$
while on $B(x_0,r)\times\{t_0-h\}$, 
$$  \bar u_\mu(x,t)-v(x,t) = u_\mu(x,t)-v(x,t)+(1-\mu)^2\big(t-t_0-|x-x_0|^2\big)\leq M_\mu -(1-\mu)^2 h\; .$$
This gives a contradiction. 

We can argue in the same way if $x_0 \in \Omega_1$ or $x_0 \in \Omega_2$ : in fact this is even easier since we may choose $r$ such that either $\overline B(x_0,r)\subset \Omega_1$ or $\overline B(x_0,r)\subset \Omega_2$; with this choice we only deal with classical Hamilton-Jacobi Equations without discontinuities and we have just to apply classical results.

The contradiction shows that $M\leq 0$ and the proof is complete.
\end{proof}

As a consequence, we have the following
\begin{theorem}\label{thm:Um.Up}
    Assume \hyp{g}{}, \hyp{\Omega}{} and \hyp{\rm C}{}. Then\\
     (i) The value function $\VFm$ is continuous and the unique solution of 
 \begin{align}
	u_t + \HHm(x,t, Du) &= 0  \hbox{ in }  \R^N \times (0,T)\;,\label{eqn-minus}\\
   u(x,0) &= g(x)  \hbox{ in }  \R^N\;.\label{eqnid}
\end{align}
(ii) $\VFm$ is the minimal supersolution of \eqref{Bellman-Om}-\eqref{Bellman-H-sub}-\eqref{Bellman-H-sup}-\eqref{eqnid}. The value function $\VFp$ is also continuous
and the maximal subsolution of \eqref{Bellman-Om}-\eqref{Bellman-H-sub}-\eqref{Bellman-H-sup}-\eqref{eqnid}.
\end{theorem}

\begin{proof}The proof of (i) is a direct consequence of Theorem~\ref{teo:HJ} and \ref{thm:comp.global} : indeed $(\VFm)^*$ and $(\VFm)_*$ are respectively sub and supersolution of \eqref{eqn-minus} by Theorem~\ref{teo:HJ} and since $(\VFm)^*(x,0)=(\VFm)_*(x,0)=g(x)$ in $\R^N$, Theorem~\ref{thm:comp.global} implies that $(\VFm)^*\leq (\VFm)_*$ in $\R^N\times [0,T]$, which implies that $\VFm$ is continuous because $(\VFm)_* \leq \VFm \leq (\VFm)^*$ in $\R^N\times [0,T]$ and therefore $(\VFm)_* = \VFm = (\VFm)^*$ in $\R^N\times [0,T]$. As a consequence $\VFm$ being both upper and lower semicontinuous, it is continuous. The uniqueness is a direct consequence of Theorem~\ref{thm:comp.global}.

For (ii), the first part is also a direct consequence of Theorem~\ref{thm:comp.global} since any supersolution of \eqref{Bellman-Om}-\eqref{Bellman-H-sub}-\eqref{Bellman-H-sup}-\eqref{eqnid} is a supersolution of \eqref{eqn-minus}-\eqref{eqnid}.

Finally, for $\VFp$, we follow the same idea as for $\VFm$ above and of \cite{BBC1} : if $u$ is a subsolution of \eqref{Bellman-Om}-\eqref{Bellman-H-sub}-\eqref{Bellman-H-sup}-\eqref{eqnid}, then by Theorem~\ref{teo:sotto}, it satisfies
$$ u_t  + \HTreg (x,t,Du)\leq 0 \quad \hbox{on  }\H\; ,$$
and in order to compare it with the supersolution $(\VFp)_*$, we use Theorem~\ref{teo:condplus.VFp} (instead of Theorem~\ref{teo:condplus} for the supersolutions in the case of $\HHm$) together with the regularization of the appendix (done on $\HHp$ and not $\HHm$). We skip the details since it is a straightforward adaptation of the proof of Theorems~\ref{thm:main.comp}-\ref{thm:comp.global}.

Notice that, as a consequence, we have $(\VFp)^* \leq (\VFp)_*$ in  $\R^N\times [0,T]$ since $(\VFp)^*$ is a subsolution of \eqref{Bellman-Om}-\eqref{Bellman-H-sub}-\eqref{Bellman-H-sup}-\eqref{eqnid},
which implies the continuity of $\VFp$.
\end{proof} 
\begin{remark} We emphasize the key role of Theorem~\ref{teo:condplus.VFp}: $\VFp$ is the only supersolution of the $\HHp$-equation for which we have such a property and this is why we do not have a complete comparison result for this equation (contrary to the $\HHm$ one).
\end{remark}

\section{Stability}
\label{sect:stability}

In  this section we prove stability results when we have a sequence of 
dynamics and costs $b_{i}^{\eps},l_{i}^{\eps}, g^{\eps}$ converging locally 
uniformly. Let us begin with a standard stability result for sub/super 
solutions.

\begin{theorem}\label{thm:main.stability}
    Assume \hyp{\Omega}{} and that, for all $\eps >0$, $b_1^\eps,b_2^\eps,
   l_1^\eps,l_2^\eps$ satisfy \hyp{\rm C}{1}-\hyp{\rm C}{3} with 
   constants uniforms in $\eps$.
   Let
   $H_i^\eps$ ($i=1,2$) and $\HT^\eps$ be defined as in \eqref{def:Ham}
   and \eqref{def:HamHT} respectively with these dynamics and costs. If
   $$\begin{aligned}
   (b_1^\eps,b_2^\eps,l_1^\eps,l_2^\eps) &\rightarrow (b_1,
   b_2,l_1,l_2) \text{ locally uniformly in } \R^N \times  [0,T] \times 
   A\,,\\
   g^{\eps} &\rightarrow g \text{ locally uniformly in }\R^{N}\,,
   \end{aligned}
    $$
   then the following holds
   \begin{itemize} 
       \item[(i)] if, for all $\eps >0$, $v_\eps$ is a lsc supersolution of  
      \begin{equation} \label{eqHHn}
	   u_t + \HHm_\eps(x,t, Du) = 0  \hbox{ in }  \R^N \times (0,T),  
       \end{equation}
then $ \underline{v}=\liminf_* v_\eps $ is a  lsc supersolution of  
\begin{equation}\label{eqHHlim} 
    u_t + \HHm(x,t, Du) = 0  \hbox{ in }  \R^N \times (0,T),  
\end{equation} where
$\HHm$ is defined as in  \eqref{def:Ham}  and \eqref{def:HamHT} through
the functions  $(b_1, b_2)$ and $(l_1, l_2)$.

\item[(ii)] If, for $\eps >0$, $u_\eps$ is an usc
subsolution of \eqref{eqHHn} and if $b_1,b_2$ satisfy
\hyp{\rm C}{4} then $\bar{u} = \limsup^*u_\eps$ is a subsolution of \eqref{eqHHlim}.
\end{itemize}
\end{theorem}

We point out the unusual form of this stability result : if for
supersolutions, the half-relaxed limit result holds true, it is not the
case anymore in general for the subsolution. This is related to the
$H_T$ inequality which sees only the subsolutions on $\H$. For exemple,
if $\H=\{x\in \R^N:\ x_N=0\}$ and if $u_\eps (x) = \sin(x_N/\eps)$, then
$\limsup^* u_\eps (x,0)\equiv 1$ on $\H$ while $u_\eps (x,0)\equiv 0$.
In this example it is clear that the $\limsup^* u_\eps $ comes from the
value of $u_\eps$ outside $\H$ and it is clear that one cannot recover
an $H_T$-inequality which sees only the values on $\H$. Assumption
\hyp{\rm C}{4} prevents these pathological situations to hold.

\begin{proof} 
This proof follows almost completely from standard  arguments  for
stability results on viscosity solutions (see, for instance \cite{Ba}):
we apply the standard stability results in $\R^N$ for the Hamiltonian
defined in the introduction, and in $\H$ for $\HT$. Since we can
flatten the boundary this last result is essentially a result in
$\R^{N-1}$.
   
The only case that need to be detailed is the proof of (ii) and more
precisely $\bar{u}$ fulfilling the inequality $u_t +\HT (x,t, Du) \leq
0$ on $\H$. To do so, we use the
\begin{lemma}\label{conv-HT}
Under the assumptions of Theorem~\ref{thm:main.stability} (ii), 
$\HT^\eps$ converges to $\HT$ locally uniformly.
\end{lemma}

We postpone the proof and return to the proof of Theorem~\ref{thm:main.stability} (ii). We first remark that, thanks to \hyp{\Omega}{}, we can argue
as  in the proof of uniqueness and suppose that we are working with
$\H=\{x_N=0\}$ (see assumption \hyp{\Omega}{x_0} and its consequences).

If $\phi\in C^1(\H\times[0,T])$ and if $(x'_0,t_0)$ is a strict local maximum point  of $ \bar{u}(y',0,s)-\phi(y',s)$ in $\H\times[0,T]$, our aim is to prove that
\begin{equation} \label{tesistab}
	\phi_t(x'_0,t_0)+\HT\big((x'_0,0),t_0,D_\H\phi(x'_0,t_0)\big)\leq0\;.
   \end{equation}
By the definition of $ \limsup^* u_\eps$, there exists a sequence $(\bar x_\eps, \bar t_\eps)$  converging to $ (x'_0,0,t_0)$ such that $\bar u (x'_0,0,t_0)=\lim_\eps u_\eps (\bar x_\eps, \bar t_\eps)$. If $(\bar x_\eps)_N \neq 0$, we set $K_\eps = |(\bar x_\eps)_N|^{-1/2}$, otherwise $K_\eps =\eps^{-1}$. Notice that $K_\eps \to +\infty $ as $\eps \to 0$.

We consider the function $\psi_\eps (x,t):=u_\eps(x,s)-\phi(x',s)- K_\eps 
|x_N|$. By classical techniques, using that $\psi_\eps (\bar x_\eps, \bar 
t_\eps) \to \bar{u}(x',0,t_0)-\phi(x',t_0)$ (this key property justifies the 
choice of $K_\eps$), one proves easily that there exists a sequence 
$(x_\eps,t_\eps)$ of maximum points of $\psi_\eps$ which converges to $ 
(x'_0,0,t_0)$.

If $x_\eps   \in \Omega_1\subset \{x\in \R^N:\ x_N>0\}$, $x\mapsto
|x_N|$ is smooth in a neighborhood of $x_\eps$ and, since $u_\eps$ is an
usc subsolution of \eqref{eqHHn}, we have $$\phi_t(x'_\eps, t_\eps)+
H_1^\eps(x_\eps,t_\e, D_\H\phi(x'_\eps, t_\eps) +K_\eps \mathbf{e}_N) \leq 0 $$
but, recalling that $K_\eps \to +\infty $ as $\eps \to 0$, this
inequality cannot hold for $\eps$ small enough because of 
\hyp{\rm C}{4}. To be more precise, since the $b_i^\eps$ converge locally uniformly to $b_i$
which statisfy \hyp{\rm C}{4}, we can take a uniform $\delta=\tilde\delta$ in  
Lemma~\ref{coer:H} which proves the claim.  

In the same way $x_\eps$
cannot be in $\Omega_2$. As a consequence,  $x_\eps$ is on $\H$ and is a
maximum point of $(y',s)\mapsto u_\eps(y',0,s)-\phi(y',s)$. But $u_\eps$
is an usc subsolution of \eqref{eqHHn}, therefore the
$\HT^\eps$-inequality holds and we conclude in the classical way using
Lemma~\ref{conv-HT}.
\end{proof} 

{\it Proof  of Lemma~\ref{conv-HT}}. By the definition of $\HT^\eps$, 
$$ \HT^\eps (x,t,p):=\sup_{A_0(x,t)} \big\{  - \psg{
    b^\eps_\H(x,t,\a)}{p}   -  l^\eps_\H(x,t,\a)   \big\}.$$
If $x \in \H$, $t\in (0,T)$ and if $(x_\eps,t_\eps)_\eps$ is a sequence in $\H\times (0,T)$ converging to $(x,t)$ and if $p_\eps\to p$, we use this definition to write
\begin{equation}\label{ineq:convHT}
 \HT^\eps (x_\eps,t_\eps,p_\eps)=  - \psg{
    b^\eps_\H(x_\eps,t_\eps,\a_\eps)}{p_\eps}   -  l^\eps_\H(x_\eps,t_\eps,\a_\eps) \geq - \psg{
    b^\eps_\H(x_\eps,t_\eps,\a)}{p_\eps}   -  l^\eps_\H(x_\eps,t_\eps,\a)   
    \end{equation}
 for any $\a \in A_0(x_\eps,t_\eps)$.
 
Again by definition, we have
 $$b^\eps_\H(x_\eps,t_\eps,\a_\eps)=\mu_\eps
    b_1 (x_\eps,t_\eps,\alpha^\eps_1) + (1-\mu_\eps)b_2(x_\eps,t_\eps,\alpha^\eps_2)\,,$$
and extracting subsequences, we can assume that $b^\eps_\H(x_\eps,t_\eps,\a_\eps)$ converges to $b_\H(x,t,\bar \a)$. In the same way, $l^\eps_\H(x_\eps,t_\eps,\a)  \to l_\H (x,t,\bar \a)$. It remains to show that  
$$ \HT (x,t,p)=- \psg{
    b_\H(x,t,\bar \a)}{p}   -  l_\H(x,t,\bar \a)\; .$$
This can be done using Inequality \eqref{ineq:convHT} and the arguments of Lemma~\ref{lemHTlip} : if
$$\HT (x,t,p)=- \psg{
    b_\H(x,t,\hat  a)}{p}   -  l_\H(x,t, \hat a)\; ,$$
we can build a sequence $\tilde{\a}_\eps\in A_0(x_\eps,t_\eps)$ such that 
$$- \psg{
    b^\eps_\H(x_\eps,t_\eps,\tilde{\a}_\eps)}{p_\eps}   -  l^\eps_\H(x_\eps,t_\eps,\tilde{\a}_\eps) \to 
       - \psg{
    b_\H(x,t,\hat \a)}{p}   -  l_\H(x,t,\hat \a)\; .$$
Passing to the limit in the inequality \eqref{ineq:convHT} with $\a=\tilde{\a}_\eps$, we have the desired conclusion.
$\Box$

We now turn to the stability of the minimal and maximal solutions.
To do so, we denote by $\Taeps$ [resp. $\Tregeps$] the set 
of admissible [resp. admissible and regular] trajectories associated to 
the dynamics $b_{i}^{\eps}\,,\ i=1,2$. We also define the costs 
functionals $J^{\eps}$ as in \eqref{def:cost}, but with $\ell^{\eps}$ and 
$g^{\eps}$.

\begin{lemma}\label{lem:limit.adm.reg}Under the assumptions of Theorem \ref{thm:main.stability}, if 
for any $\eps>0$, $(X^{\eps},a^{\eps})\in\Taeps$, the following holds
	\begin{itemize}
		\item[i)] There exists a subsequence 
		$(X^{\eps_{n}},a^{\eps_{n}})_{n}$ converging to an admissible 
		trajectory $(X,a)\in \Ta$. More precisely, $X^{\eps_{n}}\to 
		X$ uniformly in $[0,T]$  and
		$$J(x,t;(X^{\eps_{n}},a^{\eps_{n}}))\to J(x,t,(X,a)) 
		\quad\text{ uniformly in }[0,T]\,.$$
		\item[ii)] If, moreover, 
		$(X^{\eps},a^{\eps})\in\Tregeps$ for any $\eps>0$ 
		(i.e., the trajectories are regular), then we have a subsequence for which the limit 
		trajectory is also regular: $(X,a)\in \Treg$. 
		 \item[iii)] The results in  i) (and ii) ) hold true also  if we assume that for 
		each  $\eps>0$,  the trajectories $(X^{\eps},a^{\eps})\in \Taineps ( \in \Tregineps)$, 
		and we assume that $(x_\eps,t_\eps) \rightarrow (x,t)$ as $\eps \rightarrow 0$.
		
	\end{itemize}
\end{lemma}

\begin{proof} 
    The proof of $i)$ is almost standard and we only provide it for the 
    reader's convenience. On the contrary, the proof of $ii)$ reveals 
    unexpected difficulties (but which come from the particular features of the 
    control problem).
    
	\noindent\textsc{Proof of $i)$ ---} Since we want to pass to the 
	limit both on the dynamic and the cost, we rewrite the 
	differential inclusion in a different way, taking into account 
	both at the same time. 
	
	We fix $(x,t)$. Since the trajectories go backward in time,
	we introduce the variable $\sigma(s):=t-s$, starting at $\sigma(0)=t$.{}
	Then, for any $\eps>0$, using the admissible trajectory 
	$(X^{\eps},a^{\eps})$ we set 
	$$Y^{\eps}(s):=\int_{0}^{s}\ell^{\eps}\big(X^{\eps}(\tau), 
	\sigma(\tau),a^{\eps}(\tau)\big)\dtau$$
	where the Lagrangian $\ell^{\eps}$ is defined as in 
	\eqref{def:lagrangian}, but	with $l_{1}^{\eps},l_{2}^{\eps}$. In order to 
	take into acount both $X^{\eps}$ and $Y^{\eps}$ at the same time and the 
	function $\sigma(\cdot)$, we consider the mixed variable 
	$Z:=(X,Y,\sigma)\in\R^{N}\times\R\times[0,T]$, and translate the 
	differential inclusion in terms of $Z$.
		
	To do so, we use \hyp{\rm C}{3} and introduce, for 
	$i=1,2$, the sets 
	\begin{align*} 
		\mathbf{BL}_i^\eps (Z)  &:= \big\{
	    \big(b_i^\eps (X,\sigma,\alpha_i),l_i^\eps (X,\sigma,\alpha_i), -1\big) 
	    : \alpha_i \in A_i \big\}\;,\\[2mm]
	\mathcal{BL}^\eps (Z) &:= 
	\begin{cases} \mathbf{BL}_i^\eps (Z) & \text{ if } X\in\Omega_i\,,\\
	\cob \big( \mathbf{BL}_1^\eps (Z) \cup \mathbf{BL}_2^\eps (Z) 
	\big) &  \text{ if } X\in\H\,. 
    \end{cases}
	\end{align*}
	It turns out that the triple
	$Z^{\eps}:=(X^\eps,Y^\eps,\sigma)$ is a solution of 
	the differential inclusion 
	$$ \dot Z^{\eps}(s) \in \mathcal{BL}^\eps \big(Z^\eps (s)\big)\quad 
	\hbox{for a.e. }s \in [0,t) \: , \quad\text{with } Z^\eps (0)=(x,0,t)\;.$$
	 We first notice that since the $b_{i}^{\eps},l_{i}^{\eps}$ are 
	uniformly bounded, the $Z^{\eps}$ are equi-Lipschitz and equi-bounded on 
	$[0,T]$. Therefore we can extract a subsequence (denoted by $Z^{\eps_{n}}$) which converges 
	uniformly on $[0,T]$ to some $Z=(X,Y,\sigma)$. Moreover, for any 
	given $\delta>0$ and for $\eps>0$ small enough, we have, for any 
	$s\in(0,t)$
	$$\mathcal{BL}^{\eps_{n}}(Z^{\eps_{n}})\subset\mathcal{BL}(Z)+\delta 
	B_{N+2}\,,$$
	where $B_{N+2}$ is the unit ball in $\R^{N+2}$, centered at the origin. Using this information, it is immediate that $\dot Z(s)\in\mathcal{BL}\big(Z(s)\big)$ almost everywhere. In particular the limit trajectory is admissible: there exists a 
	control $a(\cdot)$ such that $\big(X,a)\in\Ta$. (See  Filippov's Lemma \cite[Theorem 8.2.1]{AF} or the proof of  Theorem 2.1 in 
	\cite{BBC1}).
	
	We deduce also that necessarily,
	$$Y^{\eps_{n}}(s)\to Y(s)= 
	\int_{0}^{s}\ell\big(X(\tau),\sigma(\tau),a(\tau)\big)\dtau \quad\text{ 
	uniformly in }[0,t]\,.$$ 
	Finally, since $g^{\eps}\to g$ locally uniformly in $\R^{N}$ and 
	$X^{\eps_{n}}\to X$ uniformly on $[0,T]$, we deduce that $J(x,t;(X^{\eps_{n}},a^{\eps_{n}}))$ converges to $J(x,t,(X,a))$ 
	 uniformly  with respect to $t\in[0,T]$.
	
	\noindent\textsc{Proof of $ii)$ ---}
	The difficulty comes from two facts: the first one is that we have to deal 
	with weak convergences in the $b_i^{\eps}, b_\H^{\eps}$-terms but the problem 
	is increased by the fact that some pieces of the trajectory $X(\cdot)$ on $\H$ 
	can be obtained as limits of trajectories $X^{\eps}(\cdot)$ which lie either on 
	$\H$, $\Omega_{1}$ or $\Omega_{2}$. In other words, the indicator functions 
	$\mathds{1}_{\{X^{\eps}\in\H\}}(\cdot)$ do not converge to 
	$\mathds{1}_{\{X\in\H\}}(\cdot)$, and similarly the
	$\mathds{1}_{\{X^{\eps}\in\Omega_{i}\}}(\cdot)$ do not converge to 
	$\mathds{1}_{\{X\in\Omega_{i}\}}(\cdot)$. We proceed in three steps.

	\noindent \textbf{Step 1.} We first recall that
	$$
		 \dot X^{\eps}(s) =  \sum_{i=1,2} b_i^{\eps}\big(X^{\eps}(s), 
		 \sigma(s),\alpha_i^{\eps}(s)\big)    
		 \mathds{1}_{\{X^{\eps}\in\Omega_{i}\}}(s)+ b_\H^{\eps} 
		 \big(X^{\eps}(s),\sigma(s),a^{\eps}(s)\big){}
		 \mathds{1}_{\{X^{\eps}\in\H\}}(s)
	$$
	converges weakly ($i.e.$ in $L^{\infty}(0,T)$ weak--$\ast$) to
	\begin{equation}\label{eq:conv.bh}
		 \dot X(s) =  \sum_{i=1,2} b_i \big(X(s), 
		 \sigma(s),\alpha_i(s)\big)    
		 \mathds{1}_{\{X\in\Omega_{i}\}}(s)+ b_\H 
		 \big(X(s),\sigma(s),a(s)\big){}
		 \mathds{1}_{\{X\in\H\}}(s)\,,
	\end{equation}
	for some control $a(\cdot)$ such that $(X,a)\in\Ta$.
	This weak convergence does not create any difficulty if $X(s)$ is in 
	$\Omega_i$ for $i=1,2$ but it is a little bit more complicated if $X(s)\in \H$ 
	since the term $b_\H \big(X(s),\sigma(s),a(s)\big)\mathds{1}_{\{X\in\H\}}(s)$ 
	is a weak limit of
	$$  \sum_{i=1,2} b_i^{\eps}\big(X^{\eps}(s), 
		 \sigma(s),\alpha_i^{\eps}(s)\big)    
		 \mathds{1}_{\{X^{\eps}\in\Omega_{i}\}}(s)\mathds{1}_{\{X\in\H\}}(s)+ b_\H^{\eps} 
		 \big(X^{\eps}(s),\sigma(s),a^{\eps}(s)\big){}
		 \mathds{1}_{\{X^{\eps}\in\H\}}(s)\mathds{1}_{\{X\in\H\}}(s)\;,
	$$
	and we have to check that both terms cannot generate singular strategies. 
	In order to examine carefully the mechanism of the weak convergence on 
	$\H$, we write, for $0\leq \tau \leq t$
	$$
		 X^{\eps}(\tau)-x =  \sum_{i=1,2}\int_{0}^{\tau}b_i^{\eps}\big(X^{\eps}(s), 
		 \sigma(s),\alpha_i^{\eps}(s)\big)    
		 \mathds{1}_{\{X^{\eps}\in\Omega_{i}\}}(s)\ds+ \int_{0}^{\tau}b_\H^{\eps} 
		 \big(X^{\eps}(s),\sigma(s),a^{\eps}(s)\big){}
		 \mathds{1}_{\{X^{\eps}\in\H\}}(s)\ds\,,
	$$
	and we use a slight modification of the procedure leading to relaxed control 
	as follows. We write
	$$
	\int_{0}^{\tau}b_1^{\eps}\big(X^{\eps}(s),\sigma(s),\alpha_1^{\eps}(s)\big) 
		 \mathds{1}_{\{X^{\eps}\in\Omega_{1}\}}(s)\ds=
		 \int_{0}^{\tau}\int_{A_{1}}b_1^{\eps} 
		 \big(X^{\eps}(s),\sigma(s),\alpha_1\big)\,
		 \nu_{1}^{\eps}(s,\dalpha_{1})\ds\,,
	$$
	where $\nu_{1}^{\eps}(s,\cdot)$ stands for the measure defined on 
	$A_{1}$ by	$\nu_{1}^{\eps}(s,E)=\delta_{\alpha_{1}^{\eps}}(E)
	\mathds{1}_{\{X^{\eps}\in\Omega_{1}\}}(s)$, for any Borelian set 
	$E\subset A_{1}$. Similarly we define $\nu_{2}^{\eps}$ and 
	$\nu_{\H}^{\eps}$ for the other terms. Notice that $\nu_{\H}^{\eps}$ is 
	a bit more complex measure since it concerns controls of the form 
	$a=(\alpha_{1},\alpha_{2},\mu)$ on $A$, but it works as for $\nu_{1}^\eps$ so we omit 
	the details. 
	
Note that, for any $s$, $\nu^{\eps}_{1}(s,A_1)+\nu^{\eps}_{2}(s,A_2)+\nu^{\eps}_{\H}(s,A)=1$ and therefore the measures $\nu_{1}^{\eps}(s,\cdot), \nu_{2}^{\eps}(s,\cdot), \nu_{\H}^{\eps}(s,\cdot)$ are uniformly bounded in $\eps$. Up to successive extractions of subsequences, they all converge weakly to some measures $\nu_{1}$, $\nu_{2}$, $\nu_{\H}$. Since the total mass is $1$, we obtain in the limit $\nu_{1}(s,A_1)+\nu_{2}(s,A_2)+\nu_{\H}(s,A)=1$.  
	Using that (also up to extraction form the proof of $i$) above), 
	$X^{\eps}$ converges uniformly on $[0,t]$ and the local uniform 
	convergence of the $b_{i}^{\eps}$, we get that
	$$
		\int_{A_{1}}b_1^{\eps} 
		 \big(X^{\eps}(s),\sigma(s),\alpha_1\big)\,
		 \nu_{1}^{\eps}(s,\dalpha_{1})
		 \ \mathop{\longrightarrow}_{\eps\to 0} \ 
		\int_{A_{1}}b_1 
		 \big(X(s),\sigma(s),\alpha_1\big)\,
		 \nu_{1}(s,\dalpha_{1}),\;\hbox{weakly in  }L^\infty(0,T)\;.
	$$
  Introducing $\pi_{1}(s):= \int_{A_1}  \nu_{1}(s,\dalpha_{1})$ and using the convexity of $A_{1}$ to gether with measurable selection argument (see \cite[Theorem 8.1.3]{AF}), the last integral can be written as 
	$b_{1}\big(X(s),\sigma(s),\alpha_{1}^{\sharp}(s)\big)\pi_{1}(s)$ for some control
	$\alpha^{\sharp}_{1}\in L^{\infty}(0,T;A_{1})$. 
	The same procedure for the 
	other two  terms provides the controls $\alpha^{\sharp}_{2}(\cdot)$, 
	$a^{\sharp}(\cdot)$ and functions $\pi_{2}(\cdot),\ \pi_{\H}(\cdot)$.
	In principle, those controls can be different from 
	$\alpha_{1}(\cdot)$, $\alpha_{2}(\cdot)$ and $a(\cdot)$ but this will not be 
	a problem since $\alpha^{\sharp}_{1}(\cdot)$, $\alpha^{\sharp}_{2}(\cdot)$, 
	$a^{\sharp}(\cdot)$ are just intermediate controls which are used to prove that the strategy $a(\cdot)$ is regular. 
	
	\noindent \textbf{Step 2.}
	We then deal with the $b_i$-terms. If $d_{\Omega_i}(x)$ denotes the distance from 
	$x$ to $\Omega_i$ then $d_{\Omega_i}(X^{\eps})$ is a sequence of Lipschitz 
	continuous functions which converges uniformly to $d_{\Omega_i}(X)$ and, up to 
	 an additional extraction of subsequence, we may assume that  the derivatives converges 
	weakly in $L^\infty$ (weak--$*$ convergence). As a consequence,
	$\frac{d}{ds}\big[d_{\Omega_i}(X^{\eps})  \big]\mathds{1}_{\{X\in\H\}}$ converges 
	weakly to 
	$\frac{d}{ds}\big[d_{\Omega_i}(X)  \big]\mathds{1}_{\{X\in\H\}}$.
	
	In order to use this convergence we have to compute $\frac{d}{ds}\big[d_{\Omega_i}(X^{\eps})  \big]$.   Using the extension of 
	$\mathbf{n}_{i}$ outside $\H$ in such a way that  $Dd_{\Omega_i}(x)=-\mathbf{n}_{i}(x)\mathds{1}_{\{ x \in\Omega_j\}}$, together with 
	 the regularity of  
	$\Omega_i$  and Stampacchia's  Theorem  we have 
	$$
	\frac{d}{ds}\big[d_{\Omega_i}(X^{\eps})  \big]=\dot X^{\eps}(s)\cdot\mathbf{n}_{i}(X^{\eps}(s))
		\mathds{1}_{\{X^{\eps}\in\Omega_j\}}(s)  \quad \mbox{ for almost all } s \in (0,T).
	$$
Indeed,  on one hand, the distance function is regular outside $\H$ while, on the other hand, $\dot X^{\eps}(s)\cdot\mathbf{n}_{i}(X^{\eps}(s))=0$ a.e. on $\H$. 
	Therefore the above convergence reads, for $i\neq j$, 
	$$ 
		\dot X^{\eps}(s)\cdot\mathbf{n}_{i}(X^{\eps}(s))
		\mathds{1}_{\{X^{\eps}\in\Omega_j\}}(s)\mathds{1}_{\{X\in\H\}}(s)	
		\longrightarrow
		\dot X (s)\cdot \mathbf{n}_{i}(X(s))\mathds{1}_{\{X\in\Omega_j\}}(s) 
		\mathds{1}_{\{X\in\H\}}(s)=0
	$$
	in $L^{\infty}(0,T)$ weak--$*$, or equivalently using the above expression of  
	$\dot X^{\eps}(s)$, 
	$$
		b_j^{\eps}\big(X^{\eps}(s), 
		\sigma(s),\alpha_j^{\eps}(s)\big)\cdot \mathbf{n}_{j}(X^{\eps}(s))\mathds{1}_{\{X^{\eps}\in\Omega_j\}}(s)
		\mathds{1}_{\{X\in\H\}}(s) \longrightarrow 0\quad\text{in } L^{\infty}(0,T)\text{ 
		weak--}*\;.
	$$
	This implies that for $i=1,2$
	\begin{equation}\label{eq:regu.dyn}
		b_i \big(X(s),\sigma(s),\alpha_i^{\sharp}(s)\big)\cdot\mathbf{n}_{i}(X(s)) \,\pi_{i}(s)=0\; \hbox{ a.e. on  }\{X(s)\in \H\}\;,
	\end{equation}
	which means that, in these terms, the involved dynamics are regular since they 
	are tangential (provided we take the $\alpha_{i}^{\sharp}$ as controls). 
	
	\noindent \textbf{Step 3.}
	We are now ready to prove that $(X,a)\in\Treg$, 
	$i.e.$ the dynamic in the $b_{\H}$-term of \eqref{eq:conv.bh} is regular.
	To do so, we introduce the convex set of regular dynamics 
	for $z\in\H$ and $0\leq s \leq t$ that we denote by  
	$$K(z,s):=\big\{b_{\H}\big(z,s,a_{*}\big)\,, 
	a_{*}\in \Aoreg(z,s)\big\}\subset\R^{N}\,.$$
	We notice that, for any $z\in\H$ and $s\in[0,T]$, $K(z,s)$ is closed 
	and convex, and the mapping $(z,s)\mapsto K(z,s)$ is continuous 
	on $\H$ for the Hausdorff distance. 
	Then, for any $\eta>0$, we consider the subset of $[0,t]$ 
	consisting of all times for which one has singular ($\eta$-enough) dynamics for 
	the control $a(\cdot)$, namely
		$$\begin{aligned}
	\Esing&:=\bigg\{s\in [0,t]: X(s)\in\H\text{ and } 
	\dist\Big(b_{\H}\big(X(s),t-s,a(s)\big);K\big(X(s),t-s\big)\Big)\geq\eta \bigg\}
	\end{aligned}$$
	and we argue by contradiction, assuming that, for some $\eta>0$, $|\Esing|>0$.
	
	If we take $s\in \Esing$, since $K(X(s),t-s)$ is closed and convex, 
	there exists an hyperplane separating $b_{\H}\big(X(s),t-s,a(s)\big)$ from 
	$K(X(s),t-s)$ and we may construct an affine function $\Psi_{s}:\R^{N}\to\R$ of the 
	form $\Psi_{s}(z)=c(s)\cdot z+d(s)$ such that
	$$\Psi_{s}\bigg(b_{\H}\big(X(s),t-s,a(s)\big)\bigg)\leq -1 \; \hbox{if  }s\in \Esing\,,\quad
	\Psi_{s}\geq+1 \text{ on }K\big(X(s),t-s\big) \,.$$
	Since the mapping $s\mapsto b_{\H}\big(X(s),t-s,a(s)\big)$ is measurable and 
	$s\mapsto K\big(X(s),t-s\big)$ is continuous (this can be seen as a consequence of Remark~\ref{remHTreg-regolare}), we can assume that the coefficients 
	$c(s),d(s)$ are in $L^\infty$ (they are bounded because the distance $\eta>0$ 
	is fixed). Hence we may consider the integral
	$$
	I^{\eps}:=\int_{0}^{t}\big(\Psi_{s}(\dot X^{\eps}(s)\big)\mathds{1}_{\Esing}(s)\ds\,.
	$$
	On the one hand, since $\Psi_{s}$ is  an affine function, by weak convergence of 
	$\dot X^{\eps}$ as $\eps\to0$ and the fact that $\dot X=b_{\H}$ when $s \in\Esing$,   we have
	$$
		I^{\eps}\rightarrow \int_{0}^{t}\Psi_{s}(\dot X(s)\big)
		\mathds{1}_{\Esing}(s)\ds=	\int_{0}^{t}\Psi_{s}
		\bigg(b_{\H}\big(X(s),t-s,a(s)\big)\bigg)
		\mathds{1}_{\Esing}(s)\ds\leq-|{\Esing}|<0\,.
	$$
	On the other hand, we can also use the decomposition
	\begin{equation}\label{Ieps.decomp}
	\begin{aligned}
	I^{\eps}
	&=\int_{0}^{t} c(s)  \mathds{1}_{\Esing}(s) \apt \sum_{i=1,2}  b_i^{\eps}  \big(X^{\eps}(s),t-s,\alpha_i^{\eps}\big)   \mathds{1}_{\{X^{\eps}\in\Omega_{i}\}}(s) \cht \ds  \\
	&+\int_{0}^{t} c(s)  \mathds{1}_{\Esing}(s) b_\H^{\eps}	(X^{\eps}(s),t-s,a^{\eps}(s) )   \mathds{1}_{\{X^{\eps}\in\H\}}(s) \ds+
	  \int_{0}^{t} d(s)  \mathds{1}_{\Esing}(s) \ds\;.
	\end{aligned}
	\end{equation}
	Notice that, in the second term above, $a^{\eps}(\cdot)$ is 
	a regular control for the trajectory $X^{\eps}$, and we want to keep this 
	property in the limit as $\eps\to0$. To do so the key remark is the following:   fix $\eps >0$ and $s \in [0,t]$ 
	for each  $a^{\eps}(s) \in A^{\rm reg}_0(X^\eps(s), t-s)$  there exists a 
	$\tilde{a}^\eps(s)   \in A^{\rm reg}_0(X(s), t-s)$ such that 
	$$
	  b_\H^{\eps}	(X^{\eps}(s),t-s,a^{\eps}(s) )-   b_\H^{\eps}(X(s),t-s,\tilde{a}^{\eps}(s) )=o_\eps(1),
	$$
	where $o_{\eps}(1)$ represents any quantity  which goes to zero as $\eps\to0$.
	Indeed,  for $\eps >0$, we can apply Remark~\ref{remHTreg-regolare} for each $s$ fixed and a 
	measurable selection argument (see Filippov's Lemma \cite[Theorem 8.2.10]{AF})  to obtain the existence of the control  $a^{\eps}(s) \in A^{\rm reg}_0(X^\eps(s), t-s)$ and then deduce the estimate by recalling that   $X^{\eps}$ 
	converges uniformly to $X$.
	Moreover, by construction and using again a measurable selection argument (see Filippov's Lemma \cite[Theorem 8.2.10]{AF}), there exists a control $a_{\star}(s) \in K(X(s),t-s)$ 
	such that 
	$$
 c(s)b_\H(X(s),t-s,a_{\star}(s)) = \min_{a\in K(X(s),t-s) }  c(s)b_\H(X(s),t-s,a).
	$$ 
	Therefore,  using the two above informations, we have
	\begin{equation}  \label{Ieps.decompmagg}
	\int_{0}^{t}  \mathds{1}_{\Esing}(s) c(s) b_\H^{\eps}	(X^{\eps}(s),t-s,a^{\eps}(s) )  \mathds{1}_{\{X^{\eps}\in\H\}}(s) \ds \geq 
	\int_{0}^{t}  \mathds{1}_{\Esing}(s)  c(s)  b_\H(X(s),t-s,a_{\star}(s))   \mathds{1}_{\{X^{\eps}\in\H\}}(s) \ds  + o_\eps(1) .
	\end{equation}

Now we can pass to the weak  limit in \eqref{Ieps.decomp}-\eqref{Ieps.decompmagg} 
	 using the measures $\nu_{i}$ and $\nu_{\H}$. We obtain
	$$\begin{aligned}
\lim_{\eps\to 0}	I^{\eps} & \geq  \int_{0}^{t} c(s)  \mathds{1}_{\Esing}(s) \apt \sum_{i=1,2}    \int_{A_{i}}
	  b_i \big(X(s),t-s,\alpha_i (s)\big) \nu_{i}(s,\dalpha_{i})   +    \int_{A}  b_\H(X(s),t-s,a_{\star}(s)) \nu_{\H}(s,d a)  \cht \ds\\
	  & + \int_{0}^{t} d(s)  \mathds{1}_{\Esing}(s) \ds  \\
	 & =  \int_{0}^{t}  \mathds{1}_{\Esing}(s)    \Psi_{s} \apt    \sum_{i=1,2}    \int_{A_{i}}
	  b_i \big(X(s),t-s,\alpha_i (s)\big) \nu_{i}(s,\dalpha_{i})   +    \int_{A}  b_\H(X(s),t-s,a_{\star}(s)) \nu_\H(s,d a)    \cht     \ds  \,.	
	  \end{aligned}
	$$
Next we remark that, by \eqref{eq:regu.dyn}, for $i=1,2$
	$$
	 \int_{A_{i}}   b_i \big(X(s),t-s,\alpha_i (s)\big) \nu_{i}(s,\dalpha_{i})  =  b_i \big(X(s),\sigma(s),\alpha_i^{\sharp}(s)\big) \pi_{i}(s) \in K(X(s),t-s)
	$$
	  and $ b_\H(X(s),t-s,a_{\star}(s))  \in  K(X(s),t-s)$ by construction.  Therefore,  since  $\nu_{1}(s,A_1)+\nu_{2}(s,A_2)+\nu_{\H}(s,A)=1$ 
	  and $K(X(s),t-s)$ is convex, 
	  we have
	  $$
	   \Psi_{s} \apt    \sum_{i=1,2}    \int_{A_{i}} b_i \big(X(s),t-s,\alpha_i\big) \nu_{i}(s,\dalpha_{i})   +    \int_{A}  b_\H(X(s),t-s,a_{\star}) \nu_\H(s,d a_{\star} )    \cht   \geq 1 
	  $$ 
   We end up with	$\lim_{\eps\to0} I^{\eps}\geq | \Esing | >0$ which is a contradiction with the fact that $\lim I^{\eps}=-|{\Esing}|<0$ by assumption. This proves 
	that for any $\eta>0$, $|\Esing|=0$ and we deduce that for almost any $s$, 
	the limit dynamic $b_{\H}\big(X(s),t-s,a(s)\big)$ is regular, which ends the proof.
	
	\noindent\textsc{Proof of $iii)$ ---} This result  follows by  remarking that the arguments above holds true  also is we consider a sequence 		$(x_\eps,t_\eps) \rightarrow (x,t)$ as $\eps \rightarrow 0$. We decided not to write it directly in the general case for the sake of simplicity.
\end{proof}

\begin{remark} Through the above proof, it can be easily seen that this stability result extends to the case when the domain depends on $\eps$ : indeed the proof is done using \hyp{\Omega}{}, reducing to the case when $\H=\{x_N=0\}$ through Assumption \hyp{\Omega}{x_0}. To extend the result, we have to suppose that the $\Omega_1^\eps, \Omega_2^\eps$ converges in a $C^1$-sense to $\Omega_1, \Omega_2$ which means that the $\Psi_\eps$ in \hyp{\Omega}{x_0} have to converge in $C^1$. Note that, this convergence  has to be  assumed $W^{2,\infty}$  if the required result is the convergence of  solutions (instead of only sub or supersolution). 

\end{remark}

Finally, we have a stability result for the maximal and minimal
solutions:

\begin{theorem}\label{thm:stab.Um.Up}
    Let us assume the hypotheses of Theorem \ref{thm:main.stability}.  
Then the associated value functions $\VFm_\eps$ and $\VFp_\eps$ 
 converge respectively to $\VFm$ and $\VFp$.
\end{theorem}
\begin{proof}
Let us first remark that the convergence  of   $\VFm_\eps$  to $\VFm$ follows classically from the stability  and comparison results   
Theorem  \ref{thm:main.stability} and Theorem  \ref{thm:Um.Up}. 
Moreover,   the same results  ensure us  that $\VFp \geq  \limsup^* \VFp_\eps$. Indeed,  we only now that $\VFp$ is  the maximal subsolution of  \eqref{eqHHlim}, therefore the stability can be applied only to the subsolutions inequality. 

 In order to conclude we need to prove that   $\VFp(x,t) \leq  \liminf^* 
 \VFp_\eps(x,t)$   for all $(x,t) \in \R^N \times [0,T]$. For 
each $\eps>0$, there exists a $(X^{\eps},a^{\eps})\in\Tregineps$ such that
$$\VFp_{\eps}(x_\eps,t_\eps)=J^{\eps}(x_\eps,t_\eps;(X^{\eps},a^{\eps}))$$
and we first consider a subsequence $(X^{\eps_{n}},a^{\eps_{n}})$ such that 
$\liminf \VFp_{\eps}(x_\eps,t_\eps)= \lim \VFp_{\eps_{n}}(x_{\eps_{n}},t_{\eps_{n}})$. Then we use 
Lemma~\ref{lem:limit.adm.reg}, parts  $iii)$:
up to another extraction, we may assume that $\VFp_{\eps_{n}}(x_{\eps_{n}},t_{\eps_{n}})=
J^{\eps_{n}}(x_{\eps_{n}},t_{\eps_{n}};(X^{\eps_{n}},a^{\eps_{n}}))\to J(x,t;(X,a))$ 
for some $(X,a)\in\Treg$. Hence,
$$\liminf \VFp_{\eps}(x_\eps,t_\eps)=J(x,t;(X,a))\geq \inf_{(X,a)\in\Treg} 
J(x,t;(X,a))=\VFp(x,t)\,,$$ which ends the proof.
\end{proof}

\section{Further Remarks and Extensions} \label{sec:FRE}

The simplified (but relevant) framework we describe above can be extended in several directions and we start by remarks concerning the different regions ($\Omega_1, \Omega_2$).

Because of the regularity assumptions we impose on the interfaces, there is no difference between \hyp{\Omega}{} and using a possibly infinite number regular open subsets $(\Omega_i)_{i}$ with either $1\leq i \leq K$ or $i\in \N$ and satisfying
the following assumptions
\begin{itemize} 
  \item[\hyp{\Omega}{'}] \textit{For all $i\neq j$,
	  $\Omega_i\cap\Omega_j=\emptyset$ and
	  $\R^N=\bigcup_{i}\overline{\Omega_i}$ ; for any
	  $z\in\H:=\R^N\setminus\Big(\bigcup_{i}
	  \Omega_i\Big)$, there exist exactly two indices $i,j$ such
	  that $z\in\overline{\Omega_i}\cap\overline{\Omega_j}:=\Gamij$. Moreover $\H:= \bigcup_{i,j}\,\Gamij$ is $C^1$ in the controllable case and $W^{2,\infty}$ in the non-controllable case,  (i.e. when there is only controllability in the normal direction)}.
\end{itemize}

Concerning the regularity assumption on $\H$, we point out that, since our key arguments are local, we are always in a two-domains framework and even in a two-mains framework with a flat interface. This is why we have chosen to present the paper with just two domains $\Omega_1$ and $\Omega_2$. On the other hand, this regularity is used through some change of variable and it is necessary in order that the transformed Hamiltonians satisfy the right assumptions to prove the comparison result. In the controllable case, the solutions are Lipschitz continuous and it could be enough to have continuous $b_i$'s and a $C^1$ change preserves this property. On the contrary, in the non-controllable case, the solutions may be just semi-continuous and the Lipschitz continuity of the $b_i$'s is necessary. Here we need a $W^{2,\infty}$ change to preserve this property.

Because of the same argument, the $\Omega_i$ may depend on $t$ and (this is an other way to formulate it) even we may assume that the $\Omega_i$ are domains in $\R^N \times (0,T)$ with the same regularity assumption as the one we use above (one has just to use \hyp{\Omega}{'} with $\R^N$ being replaced by $\R^N \times (0,T)$). This is a consequence of the fact that, through our change of variable, $t$ and the tangential coordinates on $\H$ play the same role. A corollary of this remark is that if $\mathbf{n}_i(\cdot) = (n_i^x,n_i^t) \in \R^N \times \R$ is the unit normal vector pointing outwards defined on $\partial\Omega_i$, then we have to assume $n_i^x\neq 0$. This is required to avoid, for example, the pathological situation of $\Omega_i\subset \subset \R^N \times (0,T)$.

As far as the control problem is concerned, it is clear from the proof that we can take into account without any difficulty : (i) general discount factors ($c_i(x,t,\alpha_i)$), (ii) infinite horizon control problem with multiple domains in the non-controllable case (extending the results of \cite{BBC1}) and (iii) the case where one has an additional control problem on $\H$ : here it suffices to check that the proof of Theorem~\ref{teo:condplus} (of  \cite[Thm. 3.3]{BBC1}) extends to this case. To do so, we make two remarks

(a)  The control problem on $\H$ is associated to an Hamiltonian $G$ and (\ref{condhyperSUP}) should be replaced by
		$$ \max(\phi_t(x,t)+\HT\big(x,t,D_\H\phi(x,t)\big), \phi_t(x,t)+G\big(x,t,D_\H\phi(x,t)\big)) \geq 0\; .$$
\indent (b) The proof is going to consider (in the flat boundary case)
\begin{align*}
\varphi(\delta):= &\max\{\phi_t(x,t)+ H_1(x_0,v(x_0),D_\H\phi(x_0^\prime)+ \delta e_N), \phi_t(x,t)+H_2(x_0,v(x_0),D_\H\phi(x_0^\prime)+ \delta e_N), \\
& \ \qquad \qquad\phi_t(x,t)+G\big(x,t,D_\H\phi(x,t) + \delta e_N)\}
\end{align*}
but $\phi_t(x,t)+G\big(x,t,D_\H\phi(x,t) + \delta e_N)=\phi_t(x,t)+G\big(x,t,D_\H\phi(x,t))$ since the $G$-Hamiltonian takes only into account the tangential part of the gradient and this quantity can be assumed to be strictly negative, otherwise we would be done. Therefore we see that the $G$-term plays no role in the proof.

To conclude, let us mention that the (interesting) cases of non-smooth $\H$ where the different regions can be separated by triple junction or the case of chessboard situations are still (far) out of the scope of this article.

\section{Appendix: the flat interface case}
\label{appendix}

In this appendix, we assume that we are in a local ``flat'' situation. More precisely, we denote by $\tilde{\Omega}$ a bounded open subset of $\R^N$ (we actually have in mind the image of a ball $B(x,r)$ by a diffeomorphism $\psi$ which purpose is to flatten the interface). We assume that $0\in \tilde{\Omega}$ and consider
$$\tilde\Omega_1=\{x_N>0\}\cap\tilde\Omega\,,\ \tilde\Omega_2=\{x_N<0\}\cap\tilde\Omega\,.$$
We use the notations
$\tilde \Gamma:=\partial\tilde\Omega_1\cap\partial\tilde\Omega_2=\tilde\Omega\cap\{x_N=0\}$,
so that $\tilde\Omega=\tilde\Omega_1\cup\tilde\Omega_2\cup\tilde \Gamma$.
Following Section~\ref{sect:uniqueness}, for $0<h<t_0<T$, we denote by $\tilde
Q:=\tilde\Omega\times (t_0-h,t_0)$ and $\partial_p\tilde
Q=\tilde\Omega\times\{t_0-h\}\cup\partial\tilde\Omega\times [t_0-h,t_0]$ its
parabolic boundary. We also denote by $\mathbf{e}_N$ the $N$-th unit
vector in $\R^N$.

For $i=1,2$, we are given dynamics $\tilde b_i$ and costs $\tilde l_i$
in each $\tilde\Omega_i$ and we define $\tilde H_i$, $\tilde H_T$,
$\tilde \HTreg$ exactly as we did for the same Hamiltonians without the âtildeâ.
With the convention of Section~\ref{sect:edp}, this allows us to
consider the problem
\begin{equation}\label{eqn-flat}
    \tilde w_t+ \tHHm(x,t,Dw) = 0\quad\text{in }\tilde Q\,.
\end{equation}

In all the following we assume that the dynamics and costs $\tilde
b_i,\tilde l_i$ satisfy \hyp{\rm C}{} with constants denoted with a
âtildeâ: $\tilde M_b,\ \tilde L_b,\ \tilde M_l,\ \tilde m_l$ and $\tilde
\delta$. Of course, this is the case after our reduction to the flat case 
if the $b_i$ and $l_i$ satisfy \hyp{\rm C}{}.
  Before proving the local comparison result which is the main result of
  this appendix, we need first to obtain some properties of the Hamiltonians.

\medskip

\noindent\textbf{Appendix A. Properties of the Hamiltonians}

To begin with, we prove that 
  the normal controllability assumption \hyp{\rm C}{4} gives coercivity in
  the $p_N$-variable:

 \begin{lemma}\label{coer:H}Assume that the dynamics $\tilde b_i$ and
    costs $\tilde l_i$ satisfy \hyp{\rm C}{}. Then, there exists a
    constant $\tilde C_M$ such that, for $i=1..2$ and $p=(p',p_N)$, we have
    $$ \tilde H_i(x,t,p) \geq \tilde\delta |p_N| -\tilde C_M(1+|p'|)\; ,$$ where $\delta$ is
    given by  assumption \hyp{\rm C}{4} and $\tilde C_M=\max\{ \tilde M_b,\tilde M_l\}$ in
    \hyp{\rm C}{1} and   \hyp{\rm C}{2} .
 \end{lemma}

 \begin{proof} We provide the proof in the case of $\tilde H_1$, it is
   similar for $\tilde H_2$. The (partial) controlability assumption
   \hyp{\rm C}{4} implies the existence of controls $\alpha_1, \alpha_2
   \in A_1$ such that $$-\tilde b_1(x,t,\alpha_1)\cdot
   \mathbf{e}_N=\tilde \delta>0\; ,\;-\tilde b_1(x,t,\alpha_2)\cdot
   \mathbf{e}_N=-\tilde \delta\,.$$

   Now we compute $\tilde H_1(x,t,p)$ assuming that $p_N >0$ (the other
   case is treated similarly).
 \begin{eqnarray*}
    \tilde H_1(x,t,p) &\geq &-\tilde b_1(x,t,\alpha_1)\cdot p -\tilde{l}_1(x,t,\alpha_1)\\
    & \geq & -\tilde b_1(x,t,\alpha_1)\cdot ( p'+p_N \mathbf{e}_N) -\tilde l_1(x,t,\alpha_1)\\
    & \geq & \tilde \delta p_N  -\tilde b_1(x,t,\alpha_1)\cdot p' -  \tilde l_1(x,t,\alpha_1)\\
    & \geq & \tilde \delta p_N  - \tilde C_M |p'| -\tilde C_M\; ,
 \end{eqnarray*}
   the last line coming from the boundedness of $\tilde b_1$ and $\tilde l_1$. 
   This concludes the proof.
 \end{proof}

 Let us now give the needed regularity properties of the
  tangential Hamiltonian $H_T$. We do the proof in the non-flat case for
the sake of completeness.

\begin{lemma}   \label{lemHTlip}
\label{propHtan} 
Assume \hyp{\Omega}{} and \hyp{\rm C}{}.  
The tangential Hamiltonian defined in \eqref{def:HamHT} satisfies the
following Lipschitz property: 
Moreover, for any $(z,p_\H),(z',q_\H)\in T\H$ and $t,t' \in [0,T]$
\begin{equation}  \label{HT-lip}
  \vert \HT(z,t,p_\H)-\HT(z',t',q_\H)\vert\leq M\vert (z,t)-(z',t')\vert 
  \big(\vert p_\H\vert+\vert q_\H\vert\big) + M_b\vert p_\H-q_\H\vert+
  m(\vert (z,t)-(z',t')\vert)\; , 
\end{equation}
where, if $M_b, M_l, L_b, m_l, \delta$ are given by \hyp{\rm C}{1} and \hyp{\rm C}{2},
$$ M := (L_b+2M_b(L_b + M_bL_\mathbf{n})\delta^{-1} )\,,$$
$L_\mathbf{n}$ being the Lipschitz constant of $\mathbf{n}_1$ and
$$m (t)= (L_b+2M_l\bar C\delta^{-1} )t +m_l(t)\quad \hbox{ for }t\geq 0\,.$$
\end{lemma}

{\bf Proof.} We only recal that $p_\H$ can be considered at
the same time as a vector in $T_z\H$ (of dimension $(N-1)$) and a vector
in $\R^N$ by using $(p_\H,0)$ where the zero means
``$0\mathbf{n}_1(z)$''. Then
$\psg{P_zb_\H(z,t,a)}{p_\H}=b_\H(z,t,a)\cdot p_\H$ with a slight
abuse of notations. With this in mind, the proof easily follows from 
Lemma~\ref{lemmacont} below and standard arguments. 
$\Box$   \\

\begin{remark}
    In various proofs, we extend a test
    function from $\H$ to $\R^N$, which gives a $N$-dimensional vector
    $p=D\phi$.
    Then, to test $H_T$ we have to compute the tangential projections on 
    $\H$: $p_\H=P_{z} p$ and $q_\H=P_{z'}p$
    which of course may not be the same, reflecting the possibly
    non-flat geometry of $\H$. Hence the term $M_b|p_\H-q_\H|$
    has to be dealt with even if we start from the same vector
    $p\in\R^N$ for both points $z,z'$. 
  \end{remark}

\begin{lemma}   \label{lemmacont}
Assume \hyp{\Omega}{} and \hyp{\rm C}{}.  
For any $(z,t), (z',t') \in \H \times [0,T]$ and for each  control $a \in A_0(z,t)$, there exists a control $a' \in A_0(z',t')$ such that, if $\bar C:=L_b + M_bL_\mathbf{n}$
\begin{align*}
|\bH(z,t,a) -  \bH(z',t',a'))  | & \leq (L_b+2M_b\bar C\delta^{-1} ) | (z,t)-(z,t')  |
\\
|\lH(z,t,a) -  \lH(z',t',a'))  | & \leq 2M_l\bar C\delta^{-1}  | (z,t)-(z,t')  | +m_l(| (z,t)-(z,t')  | )\; .
\end{align*}
\end{lemma}  
\begin{proof} Let us consider a control $a \in A_0(z,t)$, i.e. $\bH(z,t,a) \cdot \nor_1(z)=0$.  Fix $(z',t') \in \H \times [0,T]$, we have two possibilities.  
If $\bH(z',t',a) \cdot \nor_1(z') =0$ the conclusion easily follows because 
$a'=a \in \A_0(z',t')$ and 
\begin{align} 
|\bH(z,t,a)- \bH(z',t',a)| &\leq  L_b | (z,t)-(z',t')  | \,, \label{lipafix1} \\
|\lH(z,t,a)- \lH(z',t',a)| &\leq  m_l(| (z,t)-(z',t')  | )\,. \label{lipafix2}
\end{align}
Otherwise $\bH(z',t',a) \cdot \nor_1(z') \neq 0$. Let us suppose, for example, that  $\bH(z',t',a) \cdot \nor_1(z) > 0$ (for the other sign the same argument will apply so we will not detail it).
We first remark that  by  \hyp{\rm C}{1}
\begin{equation} \label{stimabha}
|  \bH(z',t',a) \cdot \nor_1(z') | = |  \bH(z',t',a) \cdot \nor_1(z')-  \bH(z,t,a) \cdot \nor_1(z)| \leq   \bar C  | (z,t)-(z',t')  | 
\end{equation}
with $\bar C:= L_b + M_bL_\mathbf{n}$.
By the controllability assumption in  \hyp{\rm C}{4} there exists a control $a_1 \in A$ such that $ \bH(z',t',a_1) \cdot \nor_1(z')=- \delta \nor_1(z')$ . 
We then set 
$$
 \bar{\mu}:= \frac{ \delta  }{ \bH(z',t',a) \cdot \nor_1(z')  +  \delta  }, 
$$
since $ \bar{\mu} \in ]0,1[$,  by the  convexity  assumption in  \hyp{\rm C}{3} , the exists a control $a'$ such that
$$
\bar{\mu} ( \bH(z',t',a),\lH(z',t',a)) +(1-\bar{\mu})  ( \bH(z',t',a_1),\lH(z',t',a_1))= 
(\bH(z',t',a'),\lH(z',t',a')) .
$$
By construction  $  \bH(z',t',a') \cdot \nor_1(z') =0$, therefore $a' \in A_0(z',t')$. Moreover,  
since $$(1-\bar{\mu}) =\frac{ \bH(z',t',a) \cdot \nor_1(z') }{ \bH(z',t',a) \cdot \nor_1(z')+\delta}$$
 by  \eqref{stimabha}, we have
$$ 
| \bH(z',t',a) - \bH(z',t',a') |
\leq (1-\bar{\mu}) | \bH(z',t',a)- \bH(z',t',a_1) |  \leq  
2M_b\bar C\delta^{-1} | (z,t)-(z',t')  | \; ,
$$
and the same inequality holds for $\lH$, replacing $M_b$ by $M_l$.
Hence, thanks to  \eqref{lipafix1}-\eqref{lipafix2}, we obtain
\begin{align*}
|\bH(z,t,a) -  \bH(z',t',a'))  |  &\leq (L_b+2M_b\bar C\delta^{-1} ) | (z,t)-(z',t')  |
\\
|\lH(z,t,a) -  \lH(z',t',a'))  |  &\leq 2M_l\bar C\delta^{-1}| (z,t)-(z',t')  | +m_l(| (z,t)-(z',t')  | )\; ,
\end{align*}
and this concludes the proof.
\end{proof}

\begin{remark}  \label{remHTreg-regolare}
The results of Lemma~\ref{lemHTlip} and ~\ref{lemmacont} still hold in the case 
of $\HTreg$, changing the constants in 
\eqref{HT-lip} and in the result of Lemma~\ref{lemmacont}. The simplest way to prove it is the following : we only do it for $b_1,b_2$ but a correct argument would require a proof in $(b_1,l_1), (b_2,l_2)$. We first remark that if 
$$b_\H\big(z,t,a)=\mu b_1 (z,t,\alpha_1) + (1-\mu)b_2(z,t,\alpha_2)\,,$$
and if $| (z,t)-(z',t')  |$ is small enough, we may assume without loss of generality that, for $i=1,2$,
\begin{equation}  \label{bipositivo}
 b_i (z,t,\alpha_1)\cdot \nor_i(z) \geq 3 (L_b+2M_b\bar C\delta^{-1} )| (z,t)-(z',t')  | \; .
\end{equation}
Indeed, by the controllability assumption in  \hyp{\rm C}{4}, there exists a control $\hat \alpha_i \in A_i$ such that $ b_i(z,t,\hat \alpha_i) \cdot \nor_i(z)=\delta \nor_i (z)$. Then, by taking  $| (z,t)-(z',t')  |$  small enough,  we can always assume that  $3 (L_b+2M_b\bar C\delta^{-1} )| (z,t)-(z',t')  |$ is between  $ b_i(z,t,\hat \alpha_i) \cdot \nor_i(z)$ and $b_i (z,t,\alpha_i)  \cdot \nor_i(z)$. We can then  choose $\mu_i \in [0,1]$ such that
$$ (\mu_i b_i (z,t,\alpha_i) +  (1- \mu_i )b_i(z,t,\hat \alpha_i) )\cdot \nor_i(z)= 3 (L_b+2M_b\bar C\delta^{-1} )| (z,t)-(z',t')  |\; .$$
Finally Assumption \hyp{\rm C}{3} ensures that there exists controls $\tilde \alpha_i$ such that 
$$b_i (z,t,\tilde \alpha_i) = \mu_i b_i (z,t,\alpha_i) + (1- \mu_i )b_i(z,t,\hat \alpha_i)\; .$$
To obtain a new $b_\H\big(z,t,\tilde a)$, we choose $\tilde \mu\in [0,1]$ such that
$$ [\tilde \mu b_1 (z,t,\tilde \alpha_1) + (1-\tilde \mu)b_2(z,t,\tilde \alpha_2)]\cdot \nor_1 (z)= 0\,.$$
To conclude  we remark that  a careful examination of   the estimate on $\bar \mu$  in the proof of Lemma~\ref{lemmacont} shows that,  if we start from a control $\tilde a \in A_0^{\rm reg}(z,t)$ verifying 
 \eqref{bipositivo}  the  associated control  $\tilde a' \in A_0(z',t')$  is  in fact in $ A_0^{\rm reg}(z',t')$. 
 \end{remark}

\begin{remark}
    If the $b_i$ are only assumed to be continuous, we have similar
    estimates involving the modulus of contuity $m_b$ instead of the
    Lipschitz constant $L_b$ (as we did for the $l_i$ with $m_l$).
\end{remark}

\medskip

\noindent\textbf{Appendix B. the local comparison result}

\begin{lemma}\label{lem:flat.comp} Assume that the dynamics $\tilde b_i$ and costs $\tilde l_i$ satisfy \hyp{\rm C}{}.
If $\tilde u$ is an usc subsolution of  (\ref{eqn-flat}) and $\tilde v$ a lsc supersolution of  (\ref{eqn-flat}), then
\begin{equation}\label{ineq:flat.comp}
	\Vert(\tilde u-\tilde v)_+\Vert_{L^\infty(\tilde Q)} \leq
	\Vert(\tilde u-\tilde v)_+\Vert_{L^\infty(\partial_p \tilde Q)}\,.
    \end{equation}
\end{lemma}

\begin{proof} As in \cite{BBC1} the first steps consist in regularizing
    the subsolution. To do so, depending on the context, we write either
    $x$ or $(x',x_N)$ where $x'\in \R^{N-1}$ for a point in $\tilde
    \Omega$. Moreover, for the sake of simplicity, we will use both notations: $H(x,t,p)$ or $H(x',x_N,t,p)$. \medskip

\noindent\textsc{Step 1 ---} We first define the sup-convolution in  time and in the
$x'$-variable for $\tilde u$ as follows
\begin{equation*}
    \tilde u_\alpha(x,t):=\max_{y',t'}\Big\{\tilde u(y',x_N,t')-\exp(Kt)\left(\frac{|x'-y'|^2}{\alpha^2}+
	    \frac{|t-t'|^2}{\alpha^2}\right)\Big\}
\end{equation*}
where the maximum is taken over all $y',t'$ such that $(y',x_N,t')  \in \overline{\tilde Q}$ and where $K$ is a large positive constant to be chosen later.
By the definition of the supremum, if it is achived at $y',t'$, we have 
$$  \tilde u_\alpha(x,t)= \tilde u(y',x_N,t')-\exp(Kt)\left(\frac{|x'-y'|^2}{\alpha^2}+
	    \frac{|t-t'|^2}{\alpha^2}\right) \leq  \tilde u(x,t)\; ,$$
and therefore (since $K>0$),
$\frac{|x'-y'|^2}{\alpha^2}+\frac{|t-t'|^2}{\alpha^2}\leq 2||u||_\infty\;.$
Since we want to use viscosity inequalities for $u$ at $(y',x_N,t')$, we need these points to be in $\tilde Q$ and thanks to the above inequality, in order to do it, we have to restrict $(x,t)$ to be in 
$$\tilde Q_\alpha:=\Big\{x\in\tilde\Omega:\ \dist(x,\partial\tilde\Omega) > 
    (2||\tilde u||_\infty)^{1/2}\alpha\Big\} \times \Big(t_0-h+(2||\tilde u||_\infty)^{1/2}\alpha, 
t_0 -(2||\tilde u||_\infty)^{1/2}\alpha)\Big)\; .$$

Our result on $\tilde u_\alpha$ is the
\begin{lemma}\label{lem:flat.sup.conv}
	The Lipschitz continuous function $\tilde u_\alpha$ satisfies
	$(\tilde u_\alpha)_t+\tHHm(x, t,D\tilde
	u_\alpha)\leq m(\alpha)$ in $\tilde Q_\alpha$ for some $m(\alpha)$ converging to $0$ as $\alpha$ tends to $0$.
\end{lemma}

\begin{proof}We first remark that $\tilde u_\alpha$ is Lipschitz continuous  with respect to time $t$ and to the $x'$-variable by the classical properties of the sup-convolution. Once we know that $(\tilde u_\alpha)_t$ and $D_{x'}\tilde u_\alpha$ are bounded, the Lipschitz continuity with respect to the $x_N$-variable comes from the fact that $\tilde u_\alpha$ is a subsolution of the $\tHHm$-equation thanks to the coerciveness of the  Hamiltonian in the $p_n$-variable given by Lemma \ref{coer:H}. Indeed, by applying formally Lemma \ref{coer:H}
 $$  (\tilde u_\alpha)_t+\tilde\delta \big|\partial_{x_N}\tilde u_\alpha \big| -\tilde C_M(1+|D_{x'}\tilde u_\alpha|)\leq m(\alpha)\quad\hbox{ in  }\tilde Q_\alpha\; ,$$
 a claim which can be justified by very classical arguments.

To check that it is a subsolution of the $\tHHm$-equation, we consider a test-function $\phi$ and a point $(x,t)$ where $\tilde u_\alpha -\phi$ reaches a local maximum. Then considering a maximum in $(z,s)$ of $\tilde u_\alpha(z,s) - \phi(z,s)$ leads us to consider a maximum in $(z,s,y',t')$ of $\tilde u(y',z_N,t') - \exp(Ks)\left(
	    \frac{|z'-y'|^2}{\alpha^2}+\frac{|s-t'|^2}{\alpha^2}\right)-\phi(z,s)$.
If
$$ \tilde
    u_\alpha(x,t):=\tilde u(y',x_N,t')-\exp(Kt)\left(\frac{|x'-y'|^2}{\alpha^2}+
	    \frac{|t-t'|^2}{\alpha^2}\right)  \; , $$
(we still write $y',t'$ for the variables where the max is attained for simplicity of notations) we deduce several things : first, we have a max in $z'$ and $s$ which gives
\begin{align*}
D_{x'}\phi(x',x_N,t)& =\frac{2(y'-x')}{\alpha^2}\exp(Kt)\,,\\
\phi_t(x',x_N,t)& = 
\frac{2(t'-t)}{\alpha^2}\exp(Kt) - K \exp(Kt)\left(\frac{|x'-y'|^2}{\alpha^2}+
\frac{|t-t'|^2}{\alpha^2}\right)\, .
\end{align*}    

    Then, if $x_N >0$, we write down the viscosity inequality for $\tilde u$ and
    $\tilde H_1$, the proof being similar for $\tilde H_2$ if $x_N<0$ and $\tilde H_T$  if $x_N=0$  thanks to Lemma  \ref{lemHTlip} below.   \\
    Using
    as test function $(y',x_N,t')\mapsto \phi(x',x_N,t')+\exp(Kt)\left(\frac{|x'-y'|^2}{\alpha^2}+
	    \frac{|t-t'|^2}{\alpha^2}\right)$, we have
    \begin{equation}\label{ineq-HJ-prox}
	\frac{2(t'-t)}{\alpha^2} \exp(Kt) + \tilde H_1\bigg(y',x_N,t',\frac{2(y'-x')}{\alpha^2} \exp(Kt)+
	\partial_{x_N}\phi(x',x_N,t)\,\mathbf{e}_N\bigg)\leq 0\,.
    \end{equation}
Notice that, combining the previous results, we have
    \begin{equation*}
	\phi_t(x,t)+K\exp(Kt)\left(\frac{|x'-y'|^2}{\alpha^2}+
	    \frac{|t-t'|^2}{\alpha^2}\right) + \tilde H_1\big(
	y',x_N,t',D\phi(x,t)\big) \leq 0\,.
    \end{equation*}
In order to obtain the right inequality, we have to change $y'$ in $x'$ and $t'$ in $t$.
The only difficulty for doing it, compared to the usual arguments, is the $\partial_{x_N}\phi(x',x_N,t)$-term in (\ref{ineq-HJ-prox}) which we need to control. 
{Using the lemma for (\ref{ineq-HJ-prox}) yields}
\begin{equation} \label{stimadn}
\big|\partial_{x_N}\phi\big|\leq \tilde\delta^{-1}\left(\tilde C_M
    \Big(\frac{2|y'-x'|}{\alpha^2} \exp(Kt)+1\Big)+\frac{2|t'-t|}{\alpha^2} \exp(Kt)\right)\; .
\end{equation}
On the other hand, by the Lipschitz continuity of $\tilde b_1$ and the continuity of $\tilde l_1$,
(in \hyp{\rm C}{2}) we have
$$
|\tilde H_1(y',x_N,t',p)-\tilde H_1(x,t,p)\big|\leq \tilde L_b(|y'-x'|+|t'-t|)|p|+\tilde m_l(|y'-x'|+|t'-t|)\; .$$
Hence $\phi_t(x,t)+ \tilde H_1\big(x,t,D\phi\big) \leq r.h.s\,,$ where 
$$\begin{aligned}
r.h.s &:= -K\exp(Kt)\left(\frac{|x'-y'|^2}{\alpha^2}+
	    \frac{|t-t'|^2}{\alpha^2}\right) + \tilde L_b(|y'-x'|+|t'-t|)\Big(\frac{2|y'-x'|}{\alpha^2} \exp(Kt)+ \big|\partial_{x_N}\phi\big|\Big)\\
	&   + \tilde m_l(|y'-x'|+|t'-t|)\; .
\end{aligned}$$
Therefore, thanks to \eqref{stimadn},
\begin{eqnarray*}
r.h.s & \leq & -K\exp(Kt)\left(\frac{|x'-y'|^2}{\alpha^2}+
	    \frac{|t-t'|^2}{\alpha^2}\right) + \tilde L_b  \exp(Kt) \big(|y'-x'|+|t'-t|\big)
     \frac{2|y'-x'|}{\alpha^2}  \\
	  &&   +  \frac{\tilde L_b \exp(Kt)}{\tilde \delta} \big(|y'-x'|+|t'-t|\big) 
	  \left(\tilde C_M \frac{2|y'-x'|}{\alpha^2} +\frac{2|t'-t|}{\alpha^2} \right)    \\
	& &    +   \frac { \tilde L_b \tilde C_M}{\tilde \delta} \big(|y'-x'|+|t'-t|\big)
     	+ \tilde m_b\big(|y'-x'|+|t'-t|\big)\; .
\end{eqnarray*}
Since by construction $|y'-x'|+|t'-t| \leq 2(2||\tilde u||_\infty)^{1/2}\alpha$ the last line gives the $m(\alpha)$ which appears in the
statement of Lemma~\ref{lem:flat.sup.conv}. For the other terms,  tedious but
straightforward computations and the use of Cauchy-Schwarz inequality
show that they give a negative contribution provided $K$ is big enough. And the proof of
Lemma~\ref{lem:flat.sup.conv} is complete.  
\end{proof}

\medskip

\noindent\textsc{Step 2 ---} Then, for $\eps \ll 1$, we introduce the function 
$\tilde u_\alpha^\eps:=\tilde  u_\alpha\ast\rho_\eps - [m(\alpha)+\tilde m (\eps)]t$ where $m(\alpha)$ appears in the
statement of Lemma~\ref{lem:flat.sup.conv}, $\tilde m (\eps)$ is a quantity to be chosen later which converges to $0$ when $\eps \to 0$ and  $\rho_\eps(x',t)$ is a standard (positive) mollifying kernel 
 defined on $\R^{N-1} \times [0,T]$ as follows
$$ \rho_{\eps}(x',t)=\frac{1}{\eps^{N-1}}\rho(\frac{x'}{\eps}, \frac{t}{\eps}) \; ,$$
where $\displaystyle\rho \in C^{\infty}(\R^{N-1} \times [0,T]), \int_{\R^{N-1}\times [0,T]}\rho(y)dy = 1, \text{ and } {\rm supp}\{\rho\}
=B_{\R^{N-1}\times [0,T]}(0,1).$

We assume that the support of $\rho_\eps$ is
the ball $B(0,\eps)$ so that again, we define the convolution only in
$$ \tilde Q_{\alpha,\eps}:=\big\{x\in\tilde\Omega:\ \dist(x,\partial\tilde\Omega) > (2||\tilde u||_\infty)^{1/2}\alpha +\eps\big\} \times \Big(t_0-h+(2||\tilde u||_\infty)^{1/2}\alpha+\eps,t_0
-(2||\tilde u||_\infty)^{1/2}\alpha\Big)\; .$$

\begin{lemma}\label{lem:flat.conv}
For any $\eps \ll 1$, there exists $\tilde m (\eps)$ such that $\tilde m (\eps) \to 0$ as $\eps \to 0$ and the function $\tilde u_\alpha^\eps$ satisfies
    $(\tilde u_\alpha^\eps)_t+\tHHm(x,t,D \tilde u_\alpha^\eps)\leq 0$
    in $\tilde Q_{\alpha,\eps}$.
\end{lemma}

We skip the proof of this lemma which is analogous to the corresponding one in \cite[Lemma 4.2]{BBC1} since $\tilde u_\alpha$ is Lipschitz continuous. We just point out that $\tilde m (\eps)$ comes from (and is used to control) the error in the convolution procedure.
\medskip

\noindent\textsc{Step 3 ---} We are now able to prove the comparison
result for $\tilde u$ and $\tilde v$ in $\tilde Q$.  For a fixed pair
$(\alpha,\eps)$, we have to argue in $\tilde Q_{\alpha,\eps}$. 
First, we point out that for any $\eta>0$, 
$\tilde u_\alpha^\eps - \eta t$ is  $C^1$  with respect to time $t$ and the $x_1, \dots, x_{N-1}$ variables and therefore on $\tilde \Gamma \cap \tilde Q_{\alpha,\eps}$ it
is both a test-function for the $\tilde v$-inequality and it satisfies a strict subsolution inequality in the classical sense.  
Thanks to Theorem  \ref{teo:condplus} we can argue as in \cite[Theorem 4.1]{BBC1} and conclude that  $\tilde v -(\tilde u_\alpha^\eps - \eta t)$ cannot
achieve a minimum point in $\tilde \Gamma \cap  \tilde Q_{\alpha,\eps}$. Moreover, since $\tilde u_\alpha^\eps - \eta t$  is a {\em strict} subsolution, in 
$\tilde{\Omega}_1\cap  \tilde Q_{\alpha,\eps}$ and $\tilde{\Omega}_2 \cap  \tilde Q_{\alpha,\eps}$  the conclusion follows by standard arguments since we are dealing with a standard Hamilton-Jacobi Equation.  
Thus $\tilde v -(\tilde u_\alpha^\eps - \eta t)$ cannot
have a minimum point in $\tilde Q_{\alpha,\eps}$ and this immediately
yields $$\Vert(\tilde u_\alpha^\eps -\eta t-\tilde
v)_+\Vert_{L^\infty(\tilde Q_{\alpha,\eps})} \leq \Vert(\tilde u_\alpha^\eps -\eta
t-\tilde v)_+\Vert_{L^\infty(\partial_p \tilde Q_{\alpha,\eps})}\,.$$ Letting
$\eta$ tend to $0$ we obtain $\Vert(\tilde u_\alpha^\eps -\tilde
v)_+\Vert_{L^\infty(\tilde Q_{\alpha,\eps})} \leq \Vert(\tilde u_\alpha^\eps -\tilde
v)_+\Vert_{L^\infty(\partial_p \tilde Q_{\alpha,\eps})}\,.$ In order to prove
the final result, we have to pass to the limit as $\eps\to0$ and then as
$\alpha \to 0$.

Letting $\eps$ tend to $0$ is easy since $\tilde u_\alpha$ is continuous (we may even argue in a slightly smaller domain/cylinder). Therefore
    \begin{equation*}
	\Vert(\tilde u_\alpha-m(\alpha)t-\tilde v)_+\Vert_{L^\infty(\tilde Q_\alpha)} \leq
	\Vert(\tilde u_\alpha-m(\alpha)t-\tilde v)_+\Vert_{L^\infty(\partial_p \tilde Q_\alpha)}\,.
    \end{equation*}
Fix now $\alpha_0 >0$ and $(y,s) \in \tilde Q_{\alpha_0}$. For all $0< \alpha \leq \alpha_0$ we have
 \begin{equation}  \label{stiutilde}
	(\tilde u_\alpha(y,s)-m(\alpha)t-\tilde v(y,s))_+\leq
	\Vert(\tilde u_\alpha-m(\alpha)t-\tilde v)_+\Vert_{L^\infty(\partial_p \tilde Q_\alpha)}\,.
    \end{equation}
Let us observe that by the properties of the sup-convolution and the fact that  $\tilde u$ is upper-semi-continuous 
we have that 
$\limsup_{\alpha \rightarrow 0} \Vert(\tilde u_\alpha-m(\alpha)t-\tilde v)_+\Vert_{L^\infty(\partial_p \tilde Q_\alpha)} \leq 
 \Vert(\tilde  u-\tilde v)_+\Vert_{L^\infty(\partial_p \tilde Q)}\;.
$
Therefore,  by the pointwise convergence of $\tilde u_\alpha \rightarrow \tilde u$,  passing to the limsup in  \eqref{stiutilde} we deduce  
 \begin{equation*}
 ( \tilde u (y,s) -\tilde v(y,s))_+\leq  \Vert(\tilde  u-\tilde v)_+\Vert_{L^\infty(\partial_p \tilde Q)} \quad \forall (y,s) \in  \tilde Q_{\alpha_0}. 
 \end{equation*}
Since $\alpha_0$ is arbitrary we get 
$\Vert  ( \tilde u  -\tilde v)_+  \Vert_{L^\infty(\tilde Q)}  \leq  \Vert(\tilde  u-\tilde v)_+\Vert_{L^\infty(\partial_p \tilde Q)}$
and the result is proved. 
\end{proof}



\thebibliography{}

\bibitem{AF} J-P. Aubin and H. Frankowska, Set-valued analysis. Systems \&
  Control: Foundations \& Applications, 2. BirkhÂuser Boston, Inc., Boston,
  MA, 1990.

\bibitem{ACCT} Y. Achdou, F. Camilli, A. Cutri, N. Tchou, Hamilton-Jacobi equations constrained on networks,  NoDea Nonlinear Differential Equations Appl.  20 (2013), 413--445.


\bibitem{AMV} Adimurthi, S. Mishra and G. D. Veerappa Gowda, Explicit Hopf-Lax
  type formulas for Hamilton-Jacobi equations and conservation laws with
  discontinuous coefficients. (English summary) J. Differential Equations 241
  (2007), no. 1, 1-31.


\bibitem{BCD} M. Bardi, I. Capuzzo Dolcetta, {\it Optimal control and viscosity
  solutions of Hamilton-Jacobi- Bellman equations}, Systems \& Control:
  Foundations \& Applications, Birkhauser Boston Inc., Boston, MA, 1997.

\bibitem{Ba} G. Barles,  {\it Solutions de viscosit\'e des  \'equations de
  Hamilton-Jacobi}, Springer-Verlag, Paris, 1994.

\bibitem{BBC1} G. Barles, A. Briani and E. Chasseigne, \textit{A Bellman
  approach for two-domains optimal control problems in $\R^N$},  ESAIM COCV, 19 (2013), 710-739.

\bibitem{BBCT} G. Barles, A. Briani,  E. Chasseigne and N. Tchou \textit{Homogenization Results for a Deterministic Multi-domains Periodic Control Problem}, in preparation.

\bibitem{BJ:Rate} G.~Barles and E.~R.~Jakobsen.  \newblock On the convergence
  rate of approximation schemes for Hamilton-Jacobi-Bellman equations.
  \newblock {\em M2AN Math. Model. Numer. Anal.} 36(1):33--54, 2002.


\bibitem{BP2} G. Barles and B. Perthame: {\sl Exit time problems in optimal
  control and vanishing viscosity method.} SIAM J. in Control and Optimisation,
  26, 1988, pp. 1133-1148.
  
 \bibitem{BaWo} R. Barnard and P. Wolenski: {\sl Flow Invariance on Stratified Domains.} Preprint (arXiv:1208.4742).

\bibitem{Bl1} A-P. Blanc, Deterministic exit time control problems with
  discontinuous exit costs. SIAM J. Control Optim. 35 (1997), no. 2, 399--434.


\bibitem{Bl2} A-P. Blanc, Comparison principle for the Cauchy problem for
  Hamilton-Jacobi equations with discontinuous data. Nonlinear Anal. 45 (2001),
  no. 8, Ser. A: Theory Methods, 1015--1037.

\bibitem{BrYu} A. Bressan and Y. Hong, {\it Optimal control problems on
  stratified domains}, Netw. Heterog. Media 2 (2007), no. 2, 313-331
  (electronic).

\bibitem{ScCa} F.~Camilli and 
D.~Schieborn : Viscosity solutions of Eikonal equations on topological networks,  Calc. Var. Partial Differential Equations, 46 (2013), no.3,   671--686.


\bibitem{CaSo} F Camilli and A. Siconolfi, {Time-dependent measurable
  Hamilton-Jacobi equations}, Comm. in Par. Diff. Eq. 30 (2005), 813-847.

\bibitem{Clarke} F.H. Clarke,  Optimization and nonsmooth analysis, Society of Industrial Mathematics, 1990.  

\bibitem{CR} G. Coclite and N. Risebro, Viscosity solutions of Hamilton-Jacobi
  equations with discontinuous coefficients.  J. Hyperbolic Differ. Equ. 4
  (2007), no. 4, 771--795.

\bibitem{DeZS} C. De Zan and P. Soravia, Cauchy problems for noncoercive
  Hamilton-Jacobi-Isaacs equations with discontinuous coefficients.  Interfaces
  Free Bound. 12 (2010), no. 3, 347--368.

\bibitem{DE} K. Deckelnick and C. Elliott, Uniqueness and error analysis for
  Hamilton-Jacobi equations with discontinuities.  Interfaces Free Bound. 6
  (2004), no. 3, 329--349.

\bibitem{Du} P. Dupuis, A numerical method for a calculus of variations problem
  with discontinuous integrand. Applied stochastic analysis (New Brunswick, NJ,
  1991), 90--107, Lecture Notes in Control and Inform. Sci., 177, Springer,
  Berlin, 1992.

\bibitem{Evans1989}
L.C. Evans.
\newblock The perturbed test function method for viscosity solutions of
  nonlinear {PDE}.
\newblock {\em Proc. Roy. Soc. Edinburgh Sect. A}, 111(3-4):359--375, 1989.

\bibitem{Evans1992}
L.C. Evans.
\newblock Periodic homogenisation of certain fully nonlinear partial
  differential equations.
\newblock {\em Proc. Roy. Soc. Edinburgh Sect. A}, 120(3-4):245--265, 1992.


\bibitem{Fi} A.F. Filippov, {\it Differential equations with discontinuous
  right-hand side}. Matematicheskii Sbornik,  51  (1960), pp. 99--128.  American
  Mathematical Society Translations,  Vol. 42  (1964), pp. 199--231 English
  translation Series 2.

\bibitem{fs} W.H.\ Fleming, H.M.  \ Soner, Controlled Markov
Processes and Viscosity Solutions, {\it Applications of Mathematics,}
Springer-Verlag, New York, 1993.

\bibitem{GS1} M. Garavello and P. Soravia, Optimality principles and uniqueness
  for Bellman equations of unbounded control problems with discontinuous running
  cost. NoDEA Nonlinear Differential Equations Appl. 11 (2004), no. 3, 271-298.


\bibitem{GS2} M. Garavello and P. Soravia, Representation formulas for solutions
  of the HJI equations with discontinuous coefficients and existence of value in
  differential games.  J. Optim. Theory Appl. 130 (2006), no. 2, 209-229.

\bibitem{GGR} Y. Giga, P. G\`orka and P. Rybka, A comparison principle for
  Hamilton-Jacobi equations with discontinuous Hamiltonians. Proc. Amer. Math.
  Soc. 139 (2011), no. 5, 1777-1785.
  
\bibitem{Idef} H. Ishii: 
Hamilton-Jacobi Equations with discontinuous Hamiltonians on arbitrary open
sets. Bull. Fac. Sci. Eng. Chuo Univ. {\bf 28} (1985), pp~33-77. 


\bibitem{IPer}H. Ishii :
{Perron's method for Hamilton-Jacobi
Equations.} Duke Math.  J. {\bf 55} (1987), pp~369-384.


\bibitem{L} Lions P.L. (1982) {Generalized Solutions of Hamilton-Jacobi
  Equations}, Research Notes in Mathematics 69, Pitman, Boston.

\bibitem{RaZi} Z. Rao and H. Zidani : {\it Hamilton-Jacobi-Bellman Equations on Multi-Domains}. Control and Optimization with PDE Constraints, International Series of Numerical Mathematics, vol. 164, BirkhÃ€user Basel, 2013.

\bibitem{RaSiZi} Z. Rao, A. Siconolfi and H. Zidani : Stationary Hamilton-Jacobi-Bellman
Equations on multi-domains. In preparation.


\bibitem{R} R.T.~Rockafellar, Convex analysis.  Princeton Mathematical Series,
  No. 28 Princeton University Press, Princeton, N.J. 1970 xviii+451 pp.

\bibitem{IMZ} C. Imbert, R. Monneau, and H. Zidani,{\it  A Hamilton-Jacobi approach to junction problems and application to traffic flows}. ESAIM: Control, Optimisation, and Calculus of Variations; DOI 10.1051/cocv/2012002, vol. 19(01), pp. 129--166, 2013.


\bibitem{Son} H.M. Soner, {\it Optimal control with state-space constraint} I,
  SIAM J. Control Optim. 24 (1986), no. 3, 552-561.

\bibitem{So} P. Soravia, {\it Degenerate eikonal equations with discontinuous
  refraction index}, ESAIM Control Op- tim. Calc. Var. 12 (2006).

\bibitem{Wa} T. Wasewski,  Syst\`emes de commande et \'equation au contingent,
  Bull. Acad. Pol. Sc., 9, 151-155, 1961.
\end{document}